\documentclass[a4paper]{article}

\usepackage[english]{babel}
\usepackage[utf8x]{inputenc}
\usepackage[T1]{fontenc}
\usepackage{amssymb}
\usepackage{amsthm}
\usepackage{makeidx}
\usepackage{color}
\usepackage{fullpage}
\usepackage{tikz}
\usepackage{mathtools}
\usepackage[all]{xy}
\usepackage[colorlinks=true, allcolors=blue]{hyperref}
\usepackage{comment}
\usepackage{lipsum}
\usepackage{authblk}

\usepackage[a4paper,top=3cm,bottom=2cm,left=3cm,right=3cm,marginparwidth=1.75cm]{geometry}

\usepackage{amsmath}
\usepackage{graphicx}
\usepackage[colorinlistoftodos]{todonotes}
\usepackage[colorlinks=true, allcolors=blue]{hyperref}
\usepackage[all]{xy}

\newtheorem{defi}{Definition}[section]
\newtheorem{example}[defi]{Example}
\newtheorem{teo}[defi]{Theorem}
\newtheorem{coro}[defi]{Corollary}
\newtheorem{lema}[defi]{Lemma}

\newtheorem{pro}[defi]{Proposition}
\newtheorem{rmk}[defi]{Remark}

\newcommand\blfootnote[1]{%
  \begingroup
  \renewcommand\thefootnote{}\footnote{#1}%
  \addtocounter{footnote}{-1}%
  \endgroup
}

\providecommand{\keywords}[1]{\textbf{\textit{Keywords---}} #1}

\providecommand{\MSC}[1]{\textbf{\textit{MSC---}} #1}

\title{Hopf-Galois module structure of quartic Galois extensions of $\mathbb{Q}$}
\author[1, 2]{Daniel Gil-Mu\~noz}
\author[1]{Anna Rio}

\affil[1]{\footnotesize Departament de Matem\`atiques, Universitat Polit\`ecnica de Catalunya, Edifici Omega, Jordi Girona, 1-3, 08034 Barcelona}
\affil[2]{\footnotesize Charles University, Faculty of Mathematics and Physics, Department of Algebra, Sokolovska 83, 18600 Praha 8, Czech Republic}

\date{\vspace{-5ex}}

\begin{document}

\newcommand{\Addresses}{
\bigskip
\footnotesize

D. Gil-Muñoz, \textsc{Charles University, Faculty of Mathematics and Physics, Department of Algebra, Sokolovska 83, 18600 Praha 8, Czech Republic}\par\nopagebreak
\textit{E-mail address}: \texttt{daniel.gil-munoz@mff.cuni.cz}

\medskip

A. Rio, \textsc{Departament de Matem\`atiques, Universitat Polit\`ecnica de Catalunya, Edifici Omega, Jordi Girona, 1-3, 08034 Barcelona}\par\nopagebreak
\textit{E-mail address}, \texttt{ana.rio@upc.edu}

}

\maketitle

\begin{abstract}
Given a quartic Galois extension $L/\mathbb{Q}$ of number fields and a Hopf-Galois structure $H$ on $L/\mathbb{Q}$, we study the freeness of the ring of integers $\mathcal{O}_L$ as module over the associated order $\mathfrak{A}_H$ in $H$. For the classical Galois structure $H_c$, we know by Leopoldt's theorem that $\mathcal{O}_L$ is $\mathfrak{A}_{H_c}$-free. If $L/\mathbb{Q}$ is cyclic, it admits a unique non-classical Hopf-Galois structure, whereas if it is biquadratic, it admits three such Hopf-Galois structures. In both cases, we obtain that freeness depends on the solvability in $\mathbb{Z}$ of certain generalized Pell equations. We shall translate some results on Pell equations into results on the $\mathfrak{A}_H$-freeness of $\mathcal{O}_L$.
\end{abstract}


\MSC{11R33, 16T05}.

\blfootnote{The first author was supported by Czech Science Foundation, grant 21-00420M, and by Charles University Research Centre program UNCE/SCI/022.

Email addresses: daniel.gil-munoz@mff.cuni.cz, ana.rio@upc.edu}

\section{Introduction}

Let $L/K$ be a Galois extension of number fields with group $G$. We know by Normal Basis Theorem that $L/K$ has always a normal basis, that is, a basis formed by the Galois conjugates of a single element of $L$ (see for example \cite[Theorem 3.2.12]{cohen}). Classical Galois module theory is then motivated by the problem of finding a normal basis of the ring of integers $\mathcal{O}_L$ of $L$, i.e. a normal integral basis. This amounts to determining whether $\mathcal{O}_L$ is free of rank one as $\mathcal{O}_K[G]$-module. A partial answer is provided by Noether's theorem (see \cite{noether}), which in its global form states that $\mathcal{O}_L$ is $\mathcal{O}_K[G]$-locally free if $L/K$ is (at most) tamely ramified. Here the locally freeness property of $\mathcal{O}_L$ over $\mathcal{O}_K[G]$ means that the completed valuation ring $\mathcal{O}_{L,P}\coloneqq\mathcal{O}_L\otimes_{\mathcal{O}_K}\mathcal{O}_{K,P}$ at any prime ideal $P$ of $\mathcal{O}_K$ is free as module over $\mathcal{O}_{K,P}[G]\coloneqq\mathcal{O}_K[G]\otimes_{\mathcal{O}_K}\mathcal{O}_{K,P}$.

In order to study wildly ramified extensions, Leopoldt noted that the unique $\mathcal{O}_K$-order over which $\mathcal{O}_L$ can be free is its associated order $$\mathfrak{A}_{K[G]}=\{\lambda\in K[G]\,|\,\lambda\cdot\mathcal{O}_L\subseteq\mathcal{O}_L\}$$ in $K[G]$, which by definition is the maximal $\mathcal{O}_K$-order in $K[G]$ acting on $\mathcal{O}_L$ by means of the Galois action $\cdot$. However, in general $\mathcal{O}_L$ is not free over its associated order, and determining the structure of $\mathcal{O}_L$ as module over $\mathfrak{A}_{K[G]}$ is a problem of long-standing interest. Among the most celebrated results, Leopoldt's theorem asserts that the ring of integers of an abelian extension of $\mathbb{Q}$ is free over its associated order (see \cite{leopoldt}).

In \cite{childs2}, Childs obtained results on the Galois module structure of a ring of integers based on the Hopf algebra structure of $K[G]$. This motivated the introduction of Hopf-Galois theory in the study of the module structure of $\mathcal{O}_L$, which actually generalizes the setting provided by classical Galois module theory. A Hopf-Galois structure on a finite extension of fields $L/K$ is a finite-dimensional cocommutative $K$-Hopf algebra $H$ together with a $K$-linear action $\cdot$ of $H$ on $L$ endowing $L$ with left $H$-module algebra, such that the map $$\begin{array}{rccl}
    j\colon & L\otimes_KH & \longrightarrow & \mathrm{End}_K(L) \\
     & x\otimes h & \longrightarrow & j(x\otimes h)(y)=x(h\cdot y)
\end{array}$$ is a $K$-linear isomorphism. If $L/K$ admits a Hopf-Galois structure $(H,\cdot)$, we say that $L/K$ is $H$-Galois. For a Galois extension with group $G$, $K[G]$ together with the classical Galois action on $L$ is a Hopf-Galois structure, called the classical Galois structure of $L/K$. Accordingly, any other Hopf-Galois structure $(H,\cdot)$ on $L/K$ is called non-classical.

The Hopf-Galois structures on $L/K$ can be classified and described by making use of Greither-Pareigis theorem (see \cite{greitherpareigis}). This result is valid for arbitary separable extensions, but since we are not dealing with non-Galois extensions in this paper, we will state it in the Galois case. Concretely, it establishes a one-to-one correspondence between Hopf-Galois structures on $L/K$ and subgroups of $\mathrm{Perm}(G)$ whose action on $G$ is simply transitive and that are normalized by conjugation by $G$. Here, $G$ is embedded in $\mathrm{Perm}(G)$ by means of its left regular representation $\lambda\colon G\longrightarrow\mathrm{Perm}(G)$, defined by $\lambda(\sigma)(\tau)=\sigma\tau$. If $N$ is such a subgroup of $\mathrm{Perm}(G)$, the Hopf-algebra of the corresponding Hopf-Galois structure is given by $$H\coloneqq L[N]^G=\{x\in L[N]\,|\,\lambda(\sigma)x\lambda(\sigma)^{-1}=x\hbox{ for all }\sigma\in G\},$$ i.e, the $K$-subalgebra of $L[N]$ fixed by the action of $G$ on $L[N]$, defined by the classical Galois action on $L$ and by conjugation with elements of $G$ on $N$.

The approach provided by Hopf-Galois theory in the setting of Galois module theory consists in considering the structure of $\mathcal{O}_L$ as module over its associated order $$\mathfrak{A}_H=\{h\in H\,|\,h\cdot\mathcal{O}_L\subseteq\mathcal{O}_L\}$$ in $H$, which in complete analogy with the Galois case is the maximal $\mathcal{O}_K$-order in $H$ acting on $\mathcal{O}_L$, and the unique one over which $\mathcal{O}_L$ can be free. When it is so, letting a basis of $\mathfrak{A}_H$ act on a generator of such a module yields an analog of a normal integral basis of a Galois extension. We see then that $\mathfrak{A}_H$ plays the role of $\mathfrak{A}_{K[G]}$ as a ground ring of the module structure of $\mathcal{O}_L$. Among the questions provided by this approach, one may  
consider a Galois extension $L/K$ and explore the behaviour of $\mathcal{O}_L$ as $\mathfrak{A}_H$-module as $H$ runs through the different Hopf-Galois structures on $L/K$. Research has shown that there is no general answer to this problem, see for example the papers \cite{byottp2} and \cite{childs2} or the final comment of the book \cite{childs}. 

This is the context we work with in this paper. Concretely, we consider an absolute quartic Galois extension $L/\mathbb{Q}$ of number fields and study the module structure of $\mathcal{O}_L$ over the associated order in the Hopf-Galois structures on $L/\mathbb{Q}$. By the already mentioned Leopoldt's theorem, $\mathcal{O}_L$ is indeed free over its associated order in the classical Galois structure of $L/\mathbb{Q}$, so we shall focus in the non-classical Hopf-Galois structures. The classification of these is due to \cite[Theorem 2.5]{byottp2}: if the Galois group of $L/\mathbb{Q}$ is cyclic, there is a unique non-classical Hopf-Galois structure, whereas if it is elementary abelian, there are three. We will see their explicit form in Section \ref{secthopfstr}. To study the module structure of $\mathcal{O}_L$ in both cases, we will use the techniques introduced in \cite{gilrio}. The key ingredient of this method is a matrix $M(H,L)$ that encodes full information about the action of $H$ on $L$, which we call the matrix of the action. This procedure will be called the reduction method in the sequel. As it is rather computational, it is very well suited for extensions of low degree.

In Sections \ref{sectcyclicquarticQ} and \ref{sectbiquadQ}, we will use the reduction method to study the freeness in the unique non-classical Hopf-Galois structure of cyclic quartic extensions of $\mathbb{Q}$ and in the three non-classical Hopf-Galois structures of biquadratic extensions of $\mathbb{Q}$, respectively. In both cases, we will obtain that the freeness depends on the solvability in the integers of at least one of a pair of Diophantine equations of the form $x^2-Ny^2=D$ in $\mathbb{Z}$, which are particular cases of the commonly known as generalized Pell equations. This fact was already known for tamely ramified biquadratic extensions of $\mathbb{Q}$ due to a result of Truman (see \cite[Proposition 6.1]{truman}). The connection of the freeness of $\mathcal{O}_L$ over $\mathfrak{A}_H$ with the solvability of generalized Pell equations is explored in detail in Section \ref{sectgenpell}. It is remarkable that our techniques provide results on the freeness for a global extension without necessity of knowing the behaviour of its completions. In Section \ref{sectexampleslocalglobal}, we give an example of a cyclic quartic extension $L/\mathbb{Q}$ such that $\mathcal{O}_L$ is $\mathfrak{A}_H$-locally free but not $\mathfrak{A}_H$-free.

We have omitted the huge computations in Sections \ref{sectredcyclic}, \ref{sectredbiquad1}, \ref{sectredbiquad2} and \ref{sectredbiquad3}, which consist in finding the Hermite normal form of $16\times 4$ matrices with rational coefficients (according to the definition of Hermite normal form in the paragraph following Theorem \ref{teoredmatrix}).

\section{Matrices and Galois module structure}\label{redmethod}

In this section we recall and reformulate the results in \cite[Section 4]{gilrio} to describe the Hopf-Galois module structure of an $H$-Galois extension $L/K$ such that $K$ is the fraction field of a PID. For simplicity, we will take $K=\mathbb{Q}$, so we can think of our extension just as a number field. Concretely, we establish the criterion and the subsequent procedure we will use so as to study the freeness of the ring of integers $\mathcal{O}_L$ as $\mathfrak{A}_H$-module. The results obtained can be naturally translated to the general case.

\subsection{Matrix of the action}

In this approach, we regard the elements of $H$ as the endomorphisms of $L$ induced by their action on $L$. Namely, we consider the $\mathbb{Q}$-linear representation 
$$
\begin{array}{cccc}
\rho_H\colon&H&\longrightarrow&\mathrm{End}_{\mathbb{Q}}(L)\\
&h&\longrightarrow&(x\mapsto h\cdot x)\end{array}$$ of the Hopf algebra $H$ corresponding to the structure of $L$ as $H$-module.
Note that $\rho_H$ is just the restriction of the map $j\colon H\otimes_{\mathbb{Q}} L\longrightarrow\mathrm{End}_{\mathbb{Q}}(L)$ to $H\otimes_{\mathbb{Q}}1_L\cong H$.

The main object we will consider is the matrix of $\rho_H$ as linear map. In order to define it correctly, let us fix $\mathbb{Q}$-bases $W=\{w_i\}_{i=1}^n$ and $B=\{\gamma_j\}_{j=1}^n$ of $H$ and $L$. The choice of the basis $B$ of $L$ fixes an identification of $\mathrm{End}_{\mathbb{Q}}(L)$ with $\mathcal{M}_n(\mathbb{Q})$. Let us fix the canonical basis $\{E_{ij}\}_{i,j=1}^n$ of $\mathcal{M}_n(\mathbb{Q})$, given by $E_{ij}=(\delta_{ik}\delta_{jl})_{k,l=1}^n$ for every $1\leq i,j\leq n$, where $\delta_{ab}$ is the Kronecker delta. Then, we can consider the corresponding $\mathbb{Q}$-basis $\Phi=\{\varphi_i\}_{i=1}^{n^2}$ of $\mathrm{End}_{\mathbb{Q}}(L)$, whose elements are described as follows: For every $1\leq i\leq n^2$, there are $1\leq k,j\leq n$ such that $i=k+(j-1)n$. Then, let $\varphi_i$ be the map that sends $\gamma_j$ to $\gamma_k$ and the other $\gamma_l$ to $0$.

\begin{defi} The {matrix of the action} of $H$ on $L$ with respect to the bases $W$ and $B$ is the matrix $M(H_W,L_B)$ of the linear map $\rho_H\colon H\longrightarrow\mathrm{End}_{\mathbb{Q}}(L)$ arising from the choice of the basis $W$ in $H$ and the basis $\Phi$ in $\mathrm{End}_{\mathbb{Q}}(L)$.
\end{defi}

When there is no ambiguity, we will write $M(H_W,L_B)=M(H,L)$. The definition above is completely equivalent to the one given in \cite[Definition 3.1]{gilrio}. Concretely, if we write \begin{equation}\label{blocksmatrix}
    M_j(H,L)\coloneqq\begin{pmatrix}
|& | &\dots  &|  \\
(w_1\cdot\gamma_j)_B&(w_2\cdot\gamma_j)_B&\dots &(w_n\cdot\gamma_j)_B \\
|& |&\dots & |\\
\end{pmatrix}\in \mathcal{M}_n(\mathbb{Q}),
\end{equation} for every $1\leq j\leq n$, then the matrix of the action is $$M(H,L)=\begin{pmatrix}
M_1(H,L) \\ \hline
\cdots \\ \hline
M_{n}(H,L) \end{pmatrix}\in\mathcal{M}_{n^2\times n}(\mathbb{Q}).$$

This expression is suitable to visualize the dependence of $M(H,L)$ on the basis of $H$. Indeed, if $W'$ is another $\mathbb{Q}$-basis of $H$, then $M_j(H_{W'},L_B)=M_j(H_W,L_B)P_W^{W'}$ for every $1\leq j\leq n$. With regard to the dependence on the basis of $L$, it is more convenient to use another matrix from which $M(H,L)$ can be recovered.

\begin{defi} We define the {Gram matrix} of $H$ on $L$ with respect to the bases $W$ and $B$ as the matrix $$G(H_W,L_B)=\begin{pmatrix}
w_1\cdot\gamma_1 & w_1\cdot\gamma_2 & \dots & w_1\cdot\gamma_n \\
w_2\cdot\gamma_1 & w_2\cdot\gamma_2 & \dots & w_2\cdot\gamma_n \\
\vdots & \vdots & \ddots & \vdots \\
w_n\cdot\gamma_1 & w_n\cdot\gamma_2 & \dots & w_n\cdot\gamma_n
\end{pmatrix}\in\mathcal{M}_n(L).$$
\end{defi}

Again, we will usually write $G(H,L)=G(H_W,L_B)$. Note that to recover $M(H,L)$ is enough to identify the entries with the column vectors of the coordinates of its entries with respect to $B$. Now, if $B'$ is another $\mathbb{Q}$-basis of $L$, we have the relation $G(H_W,L_{B'})=G(H_W,L_B)P_B^{B'},$ which gives a practical procedure to change the basis of $L$ in $M(H,L)$.

\subsection{Reduced matrices}

With the previous setting in mind, {let us assume that $B$ is an integral basis of $L$}. Then, the image of $\mathfrak{A}_H$ by $\rho_H$ coincides with $\mathrm{End}_{\mathbb{Z}}(\mathcal{O}_L)$. Note that this is equivalent to the statement of \cite[Theorem 3.3]{gilrio}, as the element $\rho_H(h)$ is obtained from applying the matrix $M(H,L)$ to the vector of coordinates of $h$ with respect to $W$. This means that $M(H,L)$ allows us to test membership of the associated order for elements of $H$. Now, the key idea is that we can {reduce integrally the matrix $M(H,L)$ to an invertible $n\times n$ matrix}. Namely:

\begin{teo}\label{teoredmatrix} There is a matrix $D\in\mathcal{M}_n(\mathbb{Q})$ and a unimodular matrix $U\in\mathrm{GL}_{n^2}(\mathbb{Z})$ with the property that $$UM(H,L)=\begin{pmatrix}D \\ \hline \\[-2ex] O\end{pmatrix},$$ where $O$ is the zero matrix of $\mathcal{M}_{(m-n)\times n}(\mathbb{Q})$.
\end{teo}
\begin{proof}
By \cite[Theorem 3.2]{kaplansky}, every $2\times 1$ matrix with coefficients in a PID can be reduced by means of a unimodular matrix to a matrix of the same size whose lower entry is $0$ (i.e, every PID is a Hermite ring in the sense of Kaplansky). By \cite[Theorem 3.5]{kaplansky}, the statement holds for matrices with coefficients in $\mathbb{Z}$, which is a PID. Actually, this is not enough for our purposes, since the coefficients in $M(H,L)$ lie in $\mathbb{Q}$ (not necessarily in $\mathbb{Z}$). But we can think of $M(H,L)$ as a polynomial with $n^2$ variables, so it has well defined content $c$ and primitive part $M$, and we apply the aforementioned result to the primitive part, which is a matrix with coefficients in $\mathcal{O}_L$, giving the existence of $U\in\mathrm{GL}_{n^2}(\mathbb{Z})$ and $D'\in\mathcal{M}_n(\mathbb{Z})$ such that $$UM=\begin{pmatrix}D' \\ \hline \\[-2ex] O\end{pmatrix}.$$ Now, since $M(H,L)=cM$, if we call $D=cD'\in\mathcal{M}_n(\mathbb{Q})$, then $$UM(H,L)=\begin{pmatrix}D \\ \hline \\[-2ex] O\end{pmatrix},$$ as stated.
\end{proof}

A matrix $D$ as in the previous statement will be called a \emph{reduced matrix}. Clearly, it is not unique, as multiplication by a unimodular matrix of order $n$ gives another reduced matrix. Note that \cite[Theorem 3.5]{kaplansky} also gives that we can choose $D$ to be triangular. Actually, in practice, we will prefer a concrete type of triangular matrix: the Hermite normal form of $M(H,L)$.
Since the coefficients of $M(H,L)$ are not necessarily integer numbers, we define its Hermite normal form as its content times the Hermite normal form of its primitive part. 

Another important remark is that a reduced matrix is always invertible. This follows directly from the fact that the matrix $M(H,L)$ has rank $n$, which at the same time is due to the injectivity of $\rho_H$. Now, we know that $M(H,L)$ can be reduced to an invertible matrix 
using only elementary transformations such that they and their inverses preserve the ring structure (so that their concatenation is a unimodular matrix). This is what we meant before with integral reduction. 
Moreover, the reduction being integral implies that a reduced matrix also tests membership of the associated order among elements of $H$ as $M(H,L)$ does. Using this fact, we can prove the following:

\begin{teo} A reduced matrix $D$ of $M(H,L)$ is a change of basis matrix from a basis of $\mathfrak{A}_H$ to the basis $W$.
\end{teo}

This is an alternative way to write the statement of \cite[Theorem 3.5]{gilrio}. It gives directly a constructive method to determine a basis of the associated order $\mathfrak{A}_H$ and also allows us to study the $\mathfrak{A}_H$-freeness of $\mathcal{O}_L$.

\subsection{Freeness over the associated order}\label{sectfreenessgeneral}

Let $\beta=\sum_{j=1}^n\beta_j\gamma_j\in\mathcal{O}_L$ be a potential $\mathfrak{A}_H$-generator of $\mathcal{O}_L$. In particular, $\beta$ must be a generator of $L$ as $H$-module (there is such a $\beta$ because $L$ is $H$-free of rank one, see for instance \cite[(2.16)]{childs}). In \cite[Proposition 4.2]{gilrio}, we proved that $\beta$ is a free generator if and only if the matrix $\sum_{j=1}^n\beta_jM_j(H,L)D^{-1}$ is unimodular, that is, its determinant is $1$ or $-1$. Since $\mathbb{Z}$ is a PID, it is not difficult to see that the latter is a generator of the generalized module index $[\mathcal{O}_L:\mathfrak{A}_H\cdot\beta]$, which is an ideal of $\mathbb{Z}$ (see \cite[Section II.4]{frohlichtaylor} or \cite[Section 4]{johnston}). Then, the former equivalence can be read as the fact that $\beta$ is a free generator if and only if that index is the trivial ideal.

The criterion above is enough in order to solve the problem of the freeness, but it requires a basis of $\mathfrak{A}_H$. Actually, we can provide an answer to the question of the freeness without necessity of determining explicitly such a basis (but we still need to carry out the reduction of $M(H,L)$).  
Let $\mathfrak{H}=\langle w_1,\dots, w_n\rangle_{\mathbb{Z}}$ be the $\mathbb{Z}$-module generated by the basis $W$, which is a full $\mathbb{Z}$-lattice in $H$. Note that $\mathfrak{H}$ is not necessarily contained in $\mathfrak{A}_H$.
In any case, we can consider the generalized module indexes $[\mathfrak{A}_H:\mathfrak{H}]$ and $[\mathfrak{A}_H\cdot\beta:\mathfrak{H}\cdot\beta]$. Now, we have: $$[\mathcal{O}_L:\mathfrak{H}\cdot\beta]=[\mathcal{O}_L:\mathfrak{A}_H\cdot\beta]\,[\mathfrak{A}_H\cdot\beta:\mathfrak{H}\cdot\beta],$$ and the last factor is equal to $[\mathfrak{A}_H:\mathfrak{H}]$ (this follows from the $\mathbb{Q}$-linear map $h\in H\mapsto h\cdot\beta\in H\cdot\beta$ being invertible). Recall that $D$ is the change basis matrix from a basis of $\mathfrak{A}_H$ to $W$, which is a basis of $\mathfrak{H}$. Then, \begin{equation}\label{formindex}
    [\mathfrak{A}_H:\mathfrak{H}]=\langle\mathrm{det}(D)\rangle.
\end{equation} In particular, $\mathrm{det}(D)$ does not depend on the reduced matrix $D$ up to sign (this fact may be checked independently and follows from the Hermite normal form being an invariant of the reduced matrices). We call the positive one the \emph{index of $H$ with respect to $W$}, denoted by $I_W(H,L)$. Carrying this to \eqref{formindex}, we obtain: $$[\mathcal{O}_L:\mathfrak{H}\cdot\beta]=[\mathcal{O}_L:\mathfrak{A}_H\cdot\beta]\langle I_W(H,L)\rangle.$$ Let $D_{\beta}(H,L)$ be any generator of $[\mathcal{O}_L:\mathfrak{H}\cdot\beta]$. Then:

\begin{pro} $\beta$ is a free generator of $\mathcal{O}_L$ as $\mathfrak{A}_H$-module if and only if $$|D_{\beta}(H,L)|=I_W(H,L).$$ 
\end{pro}
 
Since $B$ is chosen to be an integral basis, the number $I_W(H,L)$ is (up to sign) the determinant of any reduced matrix of $M(H_W,L_B)$. As for $D_{\beta}(H,L)$, it is the determinant of the matrix $$M_{\beta}(H,L)=\sum_{j=1}^n\beta_jM_j(H,L),$$ where $\beta=\sum_{j=1}^n\beta_j\gamma_j$. Indeed, given the expression \eqref{blocksmatrix} of the blocks $M_j(H,L)$, we see that $$M_{\beta}(H,L)=\begin{pmatrix}
|& | &\dots  &|  \\
(w_1\cdot\beta)_B&(w_2\cdot\beta)_B&\dots &(w_n\cdot\beta)_B \\
|& |&\dots & |\\
\end{pmatrix}.$$

\begin{rmk}\normalfont If for some $\beta\in\mathcal{O}_L$ we have $D_{\beta}(H,L)=0$, it means that $\beta$ is not even an $H$-generator of $L$.
\end{rmk}

To sum up, the practical procedure we use to study the $\mathfrak{A}_H$-freeness of $\mathcal{O}_L$ is as follows:

\begin{itemize}
    \item[1.] We calculate the Gram matrix $G(H,L)$, where in $H$ we fix any basis $W$ and in $L$ we fix an integral basis $B$.
    \item[2.] We build the matrix $M(H,L)$ from $G(H,L)$, and for every $\beta\in\mathcal{O}_L$, the matrix $M_{\beta}(H,L)$.
    \item[3.] We reduce the matrix $M(H,L)$ by means of a unimodular matrix to determine $I_W(H,L)$.
    \item[4.] For every $\beta\in\mathcal{O}_L$, we compute the determinant $D_{\beta}(H,L)$ of $M_{\beta}(H,L)$ in terms of the coordinates $(\beta_i)_{i=1}^n$ of $\beta$ with respect to $B$.
    \item[5.] We find conditions on the elements $\beta_i$ so that $|D_{\beta}(H,L)|=I_W(H,L)$.
\end{itemize}

In practice, we do not make explicit the second step as the specific form of the matrices $M(H,L)$ is not relevant for our purposes.

\section{Hopf-Galois structures on a quartic Galois extension}\label{secthopfstr}

The description of the Hopf-Galois structures of a quartic Galois extension $L/K$ was carried out by Byott in \cite[Theorem 2.5]{byottp2}, in the more general context of Galois extensions of degree $p^2$. We follow his approach for $p=2$. Namely, let $T=\langle\tau\rangle$ be an order $2$ subgroup of $G$ and let $\sigma\in G-T$ such that $\sigma^2=1_G$ if $G\cong C_2\times C_2$ and $\sigma^2=\tau$ otherwise. Then, we have a presentation of $G$ as follows: \begin{equation}\label{presentcyclic4}
    G=\langle\sigma,\tau\,|\,\sigma^2=\gamma,\tau^2=1_G,\sigma\tau=\tau\sigma\rangle,
\end{equation} where $\gamma=1_G$ if $G\cong C_2\times C_2$ and $\gamma=\tau$ otherwise. 

By Greither-Pareigis theorem, Hopf-Galois structures of $L/K$ are in one-to-one correspondence with regular $G$-stable subgroups of $\mathrm{Perm}(G)$. The one corresponding to the classical Galois structure is $\lambda(G)$ (this does not hold if $G$ is not abelian, see \cite[(6.10)]{childs} for the general choice). As for non-classical Hopf-Galois structures, we have:

\begin{teo}\label{teohopfstrquartic} The regular subgroups of $\mathrm{Perm}(G)$ normalized by $\lambda(G)$ other than $\lambda(G)$ are those of the form $$N_T=\langle\mu,\eta_T\rangle,$$ where $T=\langle\tau\rangle$ runs through the order $2$ subgroups of $G$ and, fixing a presentation of $G$ as in \eqref{presentcyclic4},
\begin{align*}
    \mu(\sigma^k\tau^l)&=\sigma^k\tau^{l-1}, \\
    \eta_T(\sigma^k\tau^l)&=\sigma^{k-1}\tau^{l+k-1}.
\end{align*} The action of $G$ on the previous permutations is given by $$g(\mu)=\mu\,\hbox{ for all }g\in G,\quad\sigma(\eta_T)=\mu\eta_T,\quad\tau(\eta_T)=\eta_T.$$
\end{teo}

\begin{rmk}\normalfont Actually, the classical Galois structure can be included in this classification by setting a parameter in the definition of the second generator. Concretely, it corresponds to the group $N=\langle\mu,\eta_0\rangle$ with $\eta_0(\sigma^k\tau^l)=\sigma^{k-1}\tau^l$ (regardless of the choice of $T$), and we reflect all cases in $\eta_{T,d}(\sigma^k\tau^l)=\sigma^{k-1}\tau^{l+(k-1)d}$ with $d\in\{0,1\}$.
\end{rmk}

It follows immediately from the theorem that $\mu=\lambda(\tau)=(1_G,\tau)(\sigma,\sigma\tau)$. As for the other generator, we have: 
\begin{align*}
    \eta_T&=\begin{cases}(1_G,\sigma)(\tau,\sigma\tau) & \hbox{if }G\cong C_4, \\
    (1_G,\sigma\tau,\tau,\sigma) & \hbox{if }G\cong C_2\times C_2.
    \end{cases} \\
    \eta_T^2&=\begin{cases}\mathrm{Id} & \hbox{if }G\cong C_4, \\
    \mu & \hbox{if }G\cong C_2\times C_2.
    \end{cases}
\end{align*} That is, if $G\cong C_4$, $N_T\cong C_2\times C_2$, and otherwise, $N_T\cong C_4$. 

If $G\cong C_4$, $L/K$ has a unique non-classical Hopf-Galois structure given by $N_{T,1}$. Otherwise, if $G\cong C_2\times C_2$, there are two other non-classical Hopf-Galois structures, which arise from replacing $T_1\coloneqq T$ by $T_2\coloneqq\langle\sigma\rangle$ and $\sigma$ by $\tau$ for one of them, and $T_1$ by $T_3\coloneqq\langle\sigma\tau\rangle$ (and keeping $\sigma$) for the other one. Let us determine the corresponding permutation subgroups $N_{T_i}$ for $i\in\{2,3\}$. For $i=2$, $N_{T_2}=\langle\mu_2,\eta_{T_2}\rangle$. Following the definition, $$\mu_2(\sigma^k\tau^l)=\sigma^{k-1}\tau^l,$$ whence $\mu_2=\lambda(\sigma)$. On the other hand, $$\eta_{T_2}(\sigma^k\tau^l)=\sigma^{k+l-1}\tau^{l-1},$$ so $\eta_{T_2}=(1_G,\sigma\tau,\sigma,\tau)$. Finally, for $i=3$, we have $N_{T_3}=\langle\mu_3,\eta_{T_3}\rangle$. We compute the generators: 
\begin{equation*}
    \begin{split}
        \mu_3(\sigma^k\tau^l)&=\mu_3(\sigma^{k-l}(\sigma\tau)^l)=\sigma^{k-l}(\sigma\tau)^{l-1}=\sigma^{k-1}\tau^{l-1},\\\eta_{T_3}(\sigma^k\tau^l)&=\eta_{T_3}(\sigma^{k-l}(\sigma\tau)^l)=\sigma^{k-l-1}(\sigma\tau)^{l+k-l-1}=\sigma^{-l+2(k-1)}\tau^{k-1}=\sigma^l\tau^{k-1}.
    \end{split}
\end{equation*} We deduce that $\mu_3=\lambda(\sigma\tau)=(1,\sigma\tau)(\sigma,\tau)$ and $\eta_{T_3}=(1_G,\tau,\sigma\tau,\sigma)$. \\

Once the permutation subgroups are computed, we determine the corresponding Hopf algebras of those non-classical Hopf-Galois structures. We follow \cite[Lemma 2.10]{byottp2} to find a basis of the corresponding $K$-Hopf algebra. Let $H_i$ be the $K$-Hopf algebra of the Hopf-Galois structure corresponding to $T_i$, where $i=1$ if $G$ is cyclic and $i\in\{1,2,3\}$ otherwise. Then the reference above gives $$H_i=K[\mu_i,a_{z_i}\eta_{T_i}],$$ where $z_i\in E_i\coloneqq L^{T_i}$ is such that $\sigma(z_1)=-z_1$ (resp. $\tau(z_2)=-z_2$, resp. $\sigma(z_3)=-z_3$) and $$a_{z_i}=\frac{\mathrm{Id}+\mu_i}{2}+z_i\frac{\mathrm{Id}-\mu_i}{2}.$$ Then, the second generating element is $$a_{z_i}\eta_{T_i}=\frac{\eta_{T_i}+\mu_i\eta_{T_i}+z_i(\eta_{T_i}-\mu_i\eta_{T_i})}{2}.$$ Replacing $z_i$ with $-z_i$, we get an alternative generating element, and the sum and difference of that one with the one above produce two elements of $H_i$ $$\eta_{T_i}+\mu_i\eta_{T_i},\quad z_i(\eta_{T_i}-\mu_i\eta_{T_i}),$$ such that, together with $\mathrm{Id}$ and $\mu_i$, form a $K$-basis of $H_i$. Thus, we have obtained:

\begin{pro}\label{basishopfcyclic} A cyclic quartic extension $L/K$ has two Hopf-Galois structures: the classical Galois structure, of type $C_4$ and algebra $H_c$, and a non-classical Hopf-Galois structure of type $C_2\times C_2$ with algebra $H$ having $K$-basis $$\left\lbrace\mathrm{Id},\mu,\eta_T+\mu\eta_T,z(\eta_T-\mu\eta_T)\right\rbrace,$$ where $z$ is the square root in $L$ of a non-square element in $K$.
\end{pro}

\begin{pro}\label{basishopfbiquad} Let $L/K$ be a quartic elementary abelian extension with Galois group $G$ and let $E_1/K$, $E_2/K$ and $E_3/K$ be its quadratic subextensions. The Hopf-Galois structures on $L/K$ are the classical on, of type $C_2\times C_2$ and algebra $H_c$, and three non-classical Hopf-Galois structures of type $C_4$ with algebras $\{H_i\}_{i=1}^3$ such that for every $1\leq i\leq 3$, a $K$-basis of $H_i$ is $$\left\lbrace\mathrm{Id},\mu_i,\eta_{T_i}+\mu_i\eta_{T_i},z_i(\eta_{T_i}-\mu_i\eta_{T_i})\right\rbrace,$$ where $z_i\in E_i-K$ and $z_i^2\in K$.
\end{pro}

We will sometimes say that $H_i$ is the Hopf-Galois structure given by $z_i$ or corresponding to $z_i$.

\section{Cyclic quartic extensions of $\mathbb{Q}$}\label{sectcyclicquarticQ}

Let $L/\mathbb{Q}$ be a Galois quartic extension.  By \cite[Theorem 1]{hardyetal}, $L/\mathbb{Q}$ is a cyclic quartic extension if and only if $L=\mathbb{Q}\left(\sqrt{a(d+b\sqrt{d})}\,\right),$ where: \begin{itemize}
    \item $a\in\mathbb{Z}$ is odd square-free and $b\in\mathbb{Z}_{>0}$;
    \item $d=b^2+c^2$ for some $c\in\mathbb{Z}_{>0}$ and $d$ is square-free;
    \item $\gcd(a,d)=1$.
\end{itemize}
Note that the second condition is equivalent to $d$ being the product of different primes, none of them congruent to $3\mod 4$.

\subsection{Integral bases and Hopf actions}

As already mentioned, we know by Leopoldt's theorem that $\mathcal{O}_L$ is free over its associated order in the classical Galois structure $H_c$. 
In order to study the analogous question for the non-classical Hopf-Galois structure $H$, we will follow the procedure described just at the end of Section \ref{sectfreenessgeneral}. In particular, we need the matrix of the action $M(H,L)$ where in $L$ we fix an integral basis. Therefore, we need to know an integral basis of $L$ and the action of $H$ on this basis. As for the first of these, we have the following result (see \cite{hudsonwilliams}):

\begin{teo}\label{intbasiscyclicquarticq} Let $L/\mathbb{Q}$ be a cyclic quartic extension and let $a,b,c,d\in\mathbb{Z}$ as above. Define $z=\sqrt{a(d+b\sqrt{d})}$ and $w=\sqrt{a(d-b\sqrt{d})}$. Then, an integral basis of $K$ is given as follows:
\begin{itemize}
    \item[1.] If $d\equiv0\pmod 2$, $B=\{1,\sqrt{d},z,w\}.$
    \item[2.] If $d\equiv1\pmod 2$ and $b\equiv1\pmod 2$, $B=\left\lbrace1,\frac{1+\sqrt{d}}{2},z,w\right\rbrace.$
    \item[3.] If $d\equiv1\pmod 2$, $b\equiv0\pmod 2$ and $a+b\equiv3\pmod 4$, $B=\left\lbrace1,\frac{1+\sqrt{d}}{2},\frac{z+w}{2},\frac{z-w}{2}\right\rbrace.$
    \item[4.] If $d\equiv1\pmod 2$, $b\pmod 2$, $a+b\equiv1\pmod 4$ and $a\equiv c\pmod 4$, $$B=\left\lbrace1,\frac{1+\sqrt{d}}{2},\frac{1+\sqrt{d}+z+w}{4},\frac{1-\sqrt{d}+z-w}{4}\right\rbrace.$$
    \item[5.] If $d\equiv1\pmod 2$, $b\equiv0 \pmod 2$, $a+b\equiv1\pmod 4$ and $a\equiv-c\pmod 4$, $$B=\left\lbrace1,\frac{1+\sqrt{d}}{2},\frac{1+\sqrt{d}+z-w}{4},\frac{1-\sqrt{d}+z+w}{4}\right\rbrace.$$
\end{itemize}
\end{teo}

In order to find the suitable Gram matrix we compute the matrix $G(H,L_{B_c})$ where $B_c=\{1,\sqrt{d},z,w\},$ and then we carry out the computation $G(H,L_B)=G(H,L_{B_c})P_{B_c}^B$, where $B$ is an integral basis in Theorem \ref{intbasiscyclicquarticq}. From now on, we call $B_c=\{e_1,e_2,e_3,e_4\}$ and $B=\{\gamma_1,\gamma_2,\gamma_3,\gamma_4\}$. 

The irreducible polynomial of $z$ is $f(x)=x^4-2adx^2+a^2c^2d$ and its roots are $\pm z$, $\pm w$.  
We can take $\sigma=(z,w,-z,-w)$ as generator of $G$. Let us choose $\{1_G,\sigma,\sigma^2,\sigma^3\}$ as basis for $H_c$. This gives rise to the Gram matrix $$G(H_c,L_{B_c})=\begin{pmatrix}
e_1 & e_2 & e_3 & e_4 \\
e_1 & -e_2 & e_4 & -e_3 \\
e_1 & e_2 & -e_3 & -e_4 \\
e_1 & -e_2 & -e_4 & e_3
\end{pmatrix}.$$ By Proposition \ref{basishopfcyclic}, $H$ has $\mathbb{Q}$-basis $\left\lbrace\mathrm{Id},\mu,\eta_T+\mu\eta_T,z(\eta_T-\mu\eta_T)\right\rbrace.$ We have to determine the action of this basis on $B_c$. First, $\mu=\lambda(\sigma^2)$, which acts on $L$ as $\lambda(\sigma^2)^{-1}(1_G)=\sigma^2$. Then, its action on $B_c$ is given by the third row of $G(H_c,L_{B_c})$. On the other hand, we have that $\eta_T=(1_G,\sigma)(\sigma^2,\sigma^3)$ and $\mu\eta_T=(1_G,\sigma^3)(\sigma,\sigma^2)$, which act on $L$ as $\sigma$ and $\sigma^3$ respectively. Then, the action of $\eta_T+\mu\eta_T$ on $B_c$ is given by the sum of the second and the fourth rows of $G(H_c,L_{B_c})$. We determine the action of $z(\eta_T-\mu\eta_T)$ on $B_c$ in a similar way. Therefore, the Gram matrix of $H$ where in $L$ we fix the basis $B_c$ is $$G(H,L_{B_c})=\begin{pmatrix}
e_1 & e_2 & e_3 & e_4 \\
e_1 & e_2 & -e_3 & -e_4 \\
2e_1 & -2e_2 & 0 & 0 \\
0 & 0 & 2w\sqrt{d} & -2z\sqrt{d}\end{pmatrix}.$$ 

We have $zw=ac\sqrt{d}$ and
$$\frac{1}{z}=-\frac{1}{a^2c^2d}z^3+\frac{2}{ac^2}z,\quad\sqrt{d}=\frac{1}{ab}z^2-\frac{d}{b},$$
so that
$$z\sqrt{d}=\frac{1}{ab}z^3-\frac{d}{b}z=bz+cw,\quad w\sqrt{d}=-\frac{1}{ac}z^3+\frac{2d}{c}z=cz-bw.$$ 

Hence, the Gram matrix 
becomes $$G(H,L_{B_c})=\begin{pmatrix}
e_1 & e_2 & e_3 & e_4 \\
e_1 & e_2 & -e_3 & -e_4 \\
2e_1 & -2e_2 & 0 & 0 \\
0 & 0 & 2ce_3-2be_4 & -2be_3-2ce_4\end{pmatrix}.$$
Let us observe that it is independent from the twisting element $a$, so that the freeness of $\mathcal{O}_L$ over $\mathfrak{A}_H$ will also be independent of $a$.

\subsection{The reduction step}\label{sectredcyclic}

For each of the cases in Theorem \ref{intbasiscyclicquarticq}, we compute $G(H,L_B)=G(H,L_{B_c})P_{B_c}^B$. The knowledge of the Gram matrix $G(H,L_B)$ at each case allows to construct the matrix of the action $M(H,L_B)$. Then, we may find the Hermite normal form of $M(H,L_B)$ to compute the index $I(H,L)$.

\subsubsection{Case $1$: $d\equiv0\pmod 2$}

Since $\gamma_i=e_i$ for all $i$, the Gram matrix is $$G(H,L_B)=\begin{pmatrix}
\gamma_1 & \gamma_2 & \gamma_3 & \gamma_4 \\
\gamma_1 & \gamma_2 & -\gamma_3 & -\gamma_4 \\
2\gamma_1 & -2\gamma_2 & 0 & 0 \\
0 & 0 & 2c\gamma_3-2b\gamma_4 & -2b\gamma_3-2c\gamma_4\end{pmatrix}$$ and we can reduce the matrix $M(H,L)$ to $$\begin{pmatrix}
1 & 1 & 2 & 0 \\
0 & 2 & 2 & -2c \\
0 & 0 & 4 & 0 \\
0 & 0 & 0 & 2b \\
0 & 0 & 0 & 4c
\end{pmatrix}.$$ 
Using Bézout's identity with the last two rows 
we get a nonzero fourth row with $\gcd(2b,4c)$ in the diagonal entry. Since $d$ is even and square-free, $b$ and $c$ must be odd and coprime. Namely, $\gcd(2b,4c)=2$. Then, we can reduce the third row and obtain the Hermite normal form  $$D(H,L)=\begin{pmatrix}
1 & 1 & 2 & 0 \\
0 & 2 & 2 & 0 \\
0 & 0 & 4 & 0 \\
0 & 0 & 0 & 2
\end{pmatrix}$$ and $I(H,L)=16$.

\subsubsection{Case $2$: $d\equiv1\pmod 2$ and $b\equiv1\pmod 2$}

We obtain $$G(H,L)=\begin{pmatrix}
\gamma_1 &  \gamma_2 & \gamma_3 & \gamma_4 \\
\gamma_1 &  \gamma_2 & -\gamma_3 & -\gamma_4 \\
2\gamma_1 & 2\gamma_1-2\gamma_2 & 0 & 0 \\
0 & 0 & c\gamma_3-b\gamma_4 & -b\gamma_3-c\gamma_4
\end{pmatrix}.$$ In this case, we can reduce the matrix of the action to $$\begin{pmatrix}
1 & 1 & 0 & 0 \\
0 & 2 & 0 & -2c \\
0 & 0 & 2 & 0 \\
0 & 0 & 0 & 2b \\
0 & 0 & 0 & 4c
\end{pmatrix}$$ 
and arguing as in the previous one, we obtain the Hermite normal form  $$D(H,L)=\begin{pmatrix}
1 & 1 & 0 & 0 \\
0 & 2 & 0 & 0 \\
0 & 0 & 2 & 0 \\
0 & 0 & 0 & 2
\end{pmatrix},$$ which gives $I(H,L)=8$. 

\subsubsection{Case $3$: $d\equiv1\pmod 2$, $b\equiv0\pmod 2$ and $a+b\equiv3\pmod 4$}

The Gram matrix is $$G(H,L)=\begin{pmatrix}
\gamma_1 & \gamma_2 & \gamma_3 & \gamma_4 \\
\gamma_1 & \gamma_2 & -\gamma_3 & -\gamma_4 \\
2\gamma_1 & 2\gamma_1-2\gamma_2 & 0 & 0 \\
0 & 0 & 2c\gamma_3+2b\gamma_4 & -2b\gamma_3+2c\gamma_4
\end{pmatrix}.$$ We obtain exactly the same Hermite normal form as in the previous case, so $I(H,L)=8$.

\subsubsection{Case $4$: $d\equiv 1 \pmod 2 $, $b\equiv 0\pmod 2$, $a+b\equiv 1\pmod 4$ and $a\equiv c\pmod 4$}

The Gram matrix is $$G(H,L)=\begin{pmatrix}
\gamma_1 &  \gamma_2 & \gamma_3 & \gamma_4 \\
\gamma_1 &  \gamma_2 & \gamma_2-\gamma_3 & \gamma_1-\gamma_2-\gamma_4 \\
2\gamma_1 & 2\gamma_1-2\gamma_2 & \gamma_1-\gamma_2 & \gamma_2 \\
0 & 0 & h & h' \\
\end{pmatrix},$$ where $h=-c\gamma_1+(b+c)\gamma_2-2b\gamma_3+2c\gamma_4,$ and   $h'=-b\gamma_1+(b-c)\gamma_2+2c\gamma_3+2b\gamma_4.$ 
In this case, we reduce $M(H,L)$ to 
$$\begin{pmatrix}
1 & 0 & 0 & -b \\
0 & 1 & 0 & -b \\
0 & 0 & 1 & c \\
0 & 0 & 0 & 2c \\
0 & 0 & 0 & 2b
\end{pmatrix}.$$ Since $b$ and $c$ are coprime, $b$ is even and $c$ is odd, we obtain Hermite normal form $$D(H,L)=\begin{pmatrix}
1 & 0 & 0 & 0 \\
0 & 1 & 0 & 0 \\
0 & 0 & 1 & 1 \\
0 & 0 & 0 & 2
\end{pmatrix}$$ and $I(H,L)=2$.

\subsubsection{Case $5$: $d\equiv1\pmod 2$, $b\equiv0\pmod 2$, $a+b\equiv1\pmod 4$ and $a\equiv-c\pmod 4$}

The Gram matrix is $$G(H,L)=\begin{pmatrix}
\gamma_1 &  \gamma_2 & \gamma_3 & \gamma_4 \\
\gamma_1 &  \gamma_2 & \gamma_2-\gamma_3 & \gamma_1-\gamma_2-\gamma_4 \\
2\gamma_1 & 2\gamma_1-2\gamma_2 & \gamma_1-\gamma_2 & \gamma_2 \\
0 & 0 & h & h' \\
\end{pmatrix},$$ where $h=-c\gamma_1+(-b+c)\gamma_2+2b\gamma_3+2c\gamma_4,$ and  $h'=b\gamma_1-(b+c)\gamma_2+2c\gamma_3-2b\gamma_4.$ We reduce $M(H,L)$ to $$\begin{pmatrix}
1 & 0 & 0 & -b \\
0 & 1 & 0 & -b \\
0 & 0 & 1 & c \\
0 & 0 & 0 & 2c \\
0 & 0 & 0 & 2b
\end{pmatrix}.$$ Arguing as in Case $4$, we find that the Hermite form is the same, so again $I(H,L)=2$.

\subsection{Characterizations of the freeness}\label{sectfreenesscyclic}

For each $\beta\in\mathcal{O}_L$, we compute the determinant $D_{\beta}(H,L)$ of the matrix $M_{\beta}(H,L)$. The following table summarizes the results at each of the 5 cases.

\begin{center}
\begin{tabular}{|c|c|c|}
    \hline
    Case & $I(H,L)$ & $D_{\beta}(H,L)$ \\ \hline
    $1$ & $16$ & $16\beta_1\beta_2(b\beta_3^2+2c\beta_3\beta_4-b\beta_4^2)$ \\ \hline
    $2$ & $8$ & $8\beta_2(2\beta_1+\beta_2)(b\beta_3^2+2c\beta_3\beta_4-b\beta_4^2)$ \\  \hline
    $3$ & $8$ & $-8\beta_2(2\beta_1+\beta_2)(c\beta_3^2+2b\beta_3\beta_4-c\beta_4^2)$ \\ \hline
    $4$ & $2$ & $-2(2\beta_2+\beta_3-\beta_4)(4\beta_1+2\beta_2+\beta_3+\beta_4)(c\beta_3^2+2b\beta_3\beta_4-c\beta_4^2)$ \\ \hline
    $5$ & $2$ & $-2(2\beta_2+\beta_3-\beta_4)(4\beta_1+2\beta_2+\beta_3+\beta_4)(c\beta_3^2-2b\beta_3\beta_4-c\beta_4^2)$ \\ \hline
\end{tabular}
\end{center}

We see from the table that $D_{\beta}(H,L)$ is a product of $I(H,L)$, linear polynomials on $\beta_1,\beta_2,\beta_3,\beta_4$ with coprime coefficients, and a quadratic polynomial on $\beta_3$ and $\beta_4$.

We can give a quite uniform treatment for all five cases. Recall that $\beta$ is a generator if and only if $|D_{\beta}(H,L)|=I(H,L)$.

\begin{lema} The linear factors of $D_{\beta}(H,L)$ do not add any restriction on the $\mathfrak{A}_H$-freeness of $\mathcal{O}_L$. More accurately, there are $\beta_i\in\mathbb{Z}$ such that $|D_{\beta}(H,L)|=I(H,L)$ if and only if there are $\beta_i\in\mathbb{Z}$ for which the quadratic factor is $1$ or $-1$.
\end{lema}
\begin{proof}
The left to right implication is trivial. For the first three cases, the converse is trivial as well, as there are $\beta_1,\beta_2\in\mathbb{Z}$ such that the linear factors become $\pm1$, and they do not depend on $\beta_3$ and $\beta_4$. 

We focus on cases $4$ and $5$. If we want the two linear factors to be $1$, then we obtain $$\beta_1=-\frac{\beta_4}{2},\quad\beta_2=\frac{\beta_4-\beta_3+1}{2}.$$ If we want the first one to be $1$ and the second one to be $-1$, we obtain $$\beta_1=-\frac{\beta_4+1}{2},\quad\beta_2=\frac{\beta_4-\beta_3+1}{2}.$$ At least one of these pairs $(\beta_1,\beta_2)$ lie in $\mathbb{Z}^2$ if and only if $\beta_3-\beta_4$ is odd. 

Now, the quadratic factor is $$c\beta_3^2\pm2b\beta_3\beta_4-c\beta_4^2\equiv c(\beta_3^2-\beta_4^2)\pmod 2. $$ If this is $1$ or $-1$, since $c$ is odd, we obtain that $\beta_3^2-\beta_4^2$ is odd, that is, $\beta_3-\beta_4$ is odd. Then choosing one of the pairs $(\beta_1,\beta_2)$ above, we obtain that $|D_{\beta}(H,L)|=I(H,L)$, as we wanted.
\end{proof}

Note that at least one of $b$ and $c$ must be odd. If we choose $b$ to be odd (so we exchange $b$ and $c$ in Cases 3, 4 and 5), then the last factor of $D_{\beta}(H,L)$ is the same up to sign and rearrangements of $\beta_3$ and $\beta_4$. Under this consideration, we obtain:

\begin{teo}\label{teofreenesscyclic} Let $L/\mathbb{Q}$ be a cyclic quartic extension and let $H$ be its non-classical Hopf-Galois structure. The following are equivalent:
\begin{itemize}
    \item[1.] $\mathcal{O}_L$ is $\mathfrak{A}_H$-free.
    \item[2.] The quadratic form $[b,2c,-b]$ represents $1$.
    \item[3.] The equation $x^2-dy^2=b$ is solvable in $\mathbb{Z}$ and has some solution $(x,y)$ such that $b$ divides $x-cy$.
\end{itemize}
Moreover, in that case, $\mathcal{O}_L$ has an $\mathfrak{A}_H$-generator $$\beta=\beta_1\gamma_1+\beta_2\gamma_2+\beta_3\gamma_3+\beta_4\gamma_4,$$ where the values of $\beta_i\in\mathbb{Z}$ are given in the following table,
where $x'=\frac{x-cy}b$ and $y'=\frac{y-x'+1}2$.
\begin{center}
\begin{tabular}{|c|c|c|c|c|}
\hline
    Case(s) & $\beta_1$ & $\beta_2$ & $\beta_3$ & $\beta_4$ \\ \hline
    $1$ & $1$ & $1$ & $x'$ & $y$ \\ \hline
    $2,3$ & $0$ & $1$ & $x'$ & $y$ \\ \hline
    $4$ & $-\lceil y/2\rceil$ & $y'$ & $x'$ & $y$ \\ \hline
    $5$ & $-\lceil y/2\rceil$ & $y'$ & $y$ & $x'$ \\ \hline
\end{tabular}
\end{center}
\end{teo}
\begin{proof}
By the previous lemma, it is enough to work with the quadratic factor. That is, there is some $\beta\in\mathcal{O}_L$ such that $|D_{\beta}(H,L)|=I(H,L)$ if and only if there is some $\beta_3,\beta_4\in\mathbb{Z}$ such that $$b\beta_3^2+2c\beta_3\beta_4-b\beta_4^2=\pm1.$$ We can always choose $\beta_1,\beta_2\in\{-1,1\}$, so the $\mathfrak{A}_H$-freeness of $\mathcal{O}_L$ is equivalent to the existence of $\beta_3,\beta_4\in\mathbb{Z}$ such that $$b\beta_3^2+2c\beta_3\beta_4-b\beta_4^2=s,$$ where $s\in\{-1,1\}$. Since 
$$bU^2+2cUV-bV^2=1 \iff -bU^2-2cUV+bV^2=-1
\iff b V^2+2c (-U)V-b (-U)^2=-1,$$ we have the equivalence between 1 and 2. For the equivalence between 2 and 3, it is enough to check if $bU^2+2cUV-bV^2=1$ has an integral solution. Since $$U=\frac{-cV\pm\sqrt{dV^2+b}}{b},$$ this is equivalent to the existence of integers $x,y$ such that $dy^2+b=x^2$ and $b$ divides $x-cy$.

The explicit expression of an $\mathfrak{A}_H$-generator of $\mathcal{O}_L$ at each case follows from a straightforward computation.
\end{proof}

\begin{rmk}\normalfont Since $b,c>0$, the quadratic form $q=[b,2c,-b]$ is indefinite and reduced in the sense of \cite[Definition 6.2.1.]{buchmannvollmer}. The proper cycle of $q$ (see \cite[Definition 6.10.2]{buchmannvollmer}) is obtained by the iterated application of the normalization of $q$ (see \cite[(6.11)]{buchmannvollmer}) and, by \cite[Proposition 6.10.3.]{buchmannvollmer}, it consists in all reduced quadratic forms properly equivalent to $q$. 
This gives rise to an algorithm to determine the $\mathfrak{A}_H$-freeness of $\mathcal{O}_L$.
\end{rmk}

\begin{rmk}\normalfont In order to have a free generator, 
$b$ has to be a square modulo $d$. When $b=1$ we have the trivial solution $x=1, y=0$ which provides a free generator. 

On the other hand, since $x^2-c^2y^2-b^2y^2=b$, if there is a solution, then $b$ divides the product $(x+cy)(x-cy)$. If $b$ is prime, it must divide one of the factors and from a solution of $x^2-dy^2=b$ we always get a free generator for $\mathcal{O}_L$.
\end{rmk}

\begin{example}\normalfont\label{examplecyclic} In the following table, we provide examples of application of the criteria to specific cyclic quartic extensions of $\mathbb{Q}$. By solution of a Pell equation $x^2-dy^2=b$, we mean a solution $(x,y)$ such that $x-cy$ is divisible by $b$.

\begin{center}

    \begin{tabular}{|c|c|c|c|c|}
    \hline
        {\it Case} &$L$  & Equation & Solution&Free Generator \\
         \hline
        1& $\mathbb{Q}(\sqrt{106+9\sqrt{106}})$&  $x^2-106y^2=9$&$-103,10$&$\gamma_1+\gamma_2-17\gamma_3+10\gamma_4$ \\ 
        \hline
        1& $\mathbb{Q}(\sqrt{10+3\sqrt{10}})$&$x^2-10y^2=3$&$\nexists$&$\nexists$ \\
        \hline
        1& $\mathbb{Q}(\sqrt{274+15\sqrt{274}})$&$x^2-274y^2=15$&$\nexists$&$\nexists$ \\
        \hline
        2&$\mathbb{Q}(\sqrt{13+3\sqrt{13}})$& 
        $x^2-13y^2=3$&$4,-1$&$\gamma_2+2\gamma_3-\gamma_4$ \\
        \hline
        3&$\mathbb{Q}(\sqrt{5+2\sqrt{5}})$&
        $x^2-5y^2=1$&$1,0$&$\gamma_2+\gamma_3$ \\
        \hline
        4&$\mathbb{Q}(\sqrt{39+6\sqrt{13}})$&$x^2-13y^2=3$& $4,-1$&$-\gamma_2+2\gamma_3-\gamma_4$\\
        \hline
        5&$\mathbb{Q}(\sqrt{15+6\sqrt{5}})$&
        $x^2-5y^2=1$&$1,0$&$\gamma_4$ \\
        \hline
    \end{tabular}
 \end{center}

In the second example the nonexistence derives from the fact that $3$ is not a square modulo $10$. In the third one, the equation $x^2-274y^2=15$ has integral solutions. For instance, $x=\pm17$, $y=\pm1$. But for these values we don't get $x-7y$ divisible by $15$. In fact, if we compute the cycle of the reduced indefinite binary quadratic form $[15,14,-15]$ we obtain the forms
$$[-15,16,14],[14,12 ,- 17],[-17,22,9],
[9,32 , -2],[-2,32 ,9],[9,22 ,- 17],[-17,12 ,14],
[14,16,-15].$$
Since we don't get the principal form, the form $[15,14,-15]$ does not represent $1$. Therefore, $\mathcal{O}_L$ is not $\mathfrak{A}_H$-free.
\end{example}

\section{Biquadratic extensions of $\mathbb{Q}$}\label{sectbiquadQ}

Let $L/\mathbb{Q}$ be a biquadratic extension. 
The top field is of the form $L=\mathbb{Q}(\sqrt{m},\sqrt{n}\,)$, for square-free $m,n\in\mathbb{Z}$ and we know from Proposition \ref{basishopfbiquad} that the Hopf-Galois structures on $L/\mathbb{Q}$ are in one-to-one correspondence with its three intermediate fields $E_1$, $E_2$ and $E_3$, which are $\mathbb{Q}(\sqrt{m}\,)$, $\mathbb{Q}(\sqrt{n}\,)$ and $\mathbb{Q}(\sqrt{mn}\,)$. Let us call $d=\mathrm{gcd}(m,n)$, and let $k=\frac{mn}{d^2}$. Then, the third intermediate field is $\mathbb{Q}(\sqrt{k}\,)$. 
Under this consideration, $m$, $n$ and $k$ are completely exchangeable: each of them is recovered from the other two as 
product of coprime parts.
In this section, just as in the cyclic case, we use the procedure described in Section \ref{sectfreenessgeneral} to study the $\mathfrak{A}_H$-freeness of $\mathcal{O}_L$.

\subsection{Integral bases and action}

The following result gives an integral basis of $L/\mathbb{Q}$ 
(see \cite[Exercise 2.43]{marcus}).

\begin{pro}\label{intbasisbiquadr} An integral basis $B$ of $L/\mathbb{Q}$ is given as follows:
\begin{itemize}
    \item[1.] If $m\equiv3\pmod 4$ and $n,k\equiv2\pmod 4$, $B=\left\lbrace1,\sqrt{m},\sqrt{n},\frac{\sqrt{n}+\sqrt{k}}{2}\right\rbrace.$
    \item[2.] If $m\equiv1\pmod 4$ and $n,k\equiv2\hbox{ or }3\pmod 4$, $B=\left\lbrace1,\frac{1+\sqrt{m}}{2},\sqrt{n},\frac{\sqrt{n}+\sqrt{k}}{2}\right\rbrace.$
    \item[3.] If $m,n,k\equiv1\pmod 4$, $B=\left\lbrace1,\frac{1+\sqrt{m}}{2},\frac{1+\sqrt{n}}{2},\left(\frac{1+\sqrt{m}}{2}\right)\left(\frac{1+\sqrt{k}}{2}\right)\right\rbrace.$
\end{itemize}
\end{pro}

Note that the previous cases cover all possible situations because $m$, $n$ and $k$ can be exchanged conveniently. We translate the strategy and the notation of the case $G\cong C_4$ to this one: for each non-classical Hopf-Galois structure $H$ on $L/\mathbb{Q}$ we first compute the Gram matrix $G(H,L_{B_c})$ where in $L$ we fix the basis $B_c=\{1,\sqrt{m},\sqrt{n},\sqrt{k}\}$, and then change to an integral basis $B$ of Proposition \ref{intbasisbiquadr}. We call $B_c=\{e_1,e_2,e_3,e_4\}$ and $B=\{\gamma_1,\gamma_2,\gamma_3,\gamma_4\}$. 

We fix the presentation of the Galois group $G=\langle\sigma,\tau\,|\,\sigma^2=1,\,\tau^2=1\rangle$ and call $T_1$, $T_2$, $T_3$ as in Section \ref{secthopfstr}. We assume without loss of generality that $\mathbb{Q}(\sqrt{m})=L^{\langle\tau\rangle}$, $\mathbb{Q}(\sqrt{n})=L^{\langle\sigma\rangle}$ and $\mathbb{Q}(\sqrt{k})=L^{\langle\sigma\tau\rangle}$ (otherwise we would reorder suitably $m$, $n$ and $k$). Then, the bases of the non-classical Hopf-Galois structures $H_1$, $H_2$ and $H_3$ are, respectively, $$\left\lbrace\mathrm{Id},\mu_1,\eta_{T_1}+\mu_1\eta_{T_1},\sqrt{m}(\eta_{T_1}-\mu_1\eta_{T_1})\right\rbrace,$$ $$\left\lbrace\mathrm{Id},\mu_2,\eta_{T_2}+\mu_2\eta_{T_2},\sqrt{n}(\eta_{T_2}-\mu_2\eta_{T_2})\right\rbrace,$$ $$\left\lbrace\mathrm{Id},\mu_3,\eta_{T_3}+\mu_3\eta_{T_3},\sqrt{k}(\eta_{T_3}-\mu_3\eta_{T_3})\right\rbrace.$$ 

In order to compute $G(H,L_{B_c})$, we first need the Gram matrix $G(H_c,L_{B_c})$ of the classical Galois structure, i.e. the action of $G$ on the basis $B_c$, which is  
$$\begin{array}{rclrclrcl}
    \sigma(\sqrt{m})&=&-\sqrt{m},& 
    \sigma(\sqrt{n})&=&\sqrt{n},&\sigma(\sqrt{k})&=&-\sqrt{k},\\
    \tau(\sqrt{m})&=&\sqrt{m},&\tau(\sqrt{n})&=&-\sqrt{n},&\tau(\sqrt{k})&=&-\sqrt{k},\\
    \sigma\tau(\sqrt{m})&=&-\sqrt{m},&\quad\sigma\tau(\sqrt{n})&=&-\sqrt{n},&\quad\sigma\tau(\sqrt{k})&=&\sqrt{k}.
\end{array}
$$
Therefore, we have $$G(H_c,L_{B_c})=\begin{pmatrix}
e_1 & e_2 & e_3 & e_4 \\
e_1 & -e_2 & e_3 & -e_4 \\
e_1 & e_2 & -e_3 & -e_4 \\
e_1 & -e_2 & -e_3 & e_4
\end{pmatrix}.$$ Note that we take $\{1_G,\sigma,\tau,\sigma\tau\}$ as basis for $H_c$, so the $i$-th row in $G(H_c,L_{B_c})$ corresponds to the action of the $i$-th element of this basis on $B_c$. To find the Gram matrix $G(H_i,L_{B_c})$ for $i\in\{1,2,3\}$ from $G(H_c,L_{B_c})$, we use that $\mu_i$, $\eta_{T_i}+\mu_i\eta_{T_i}$ and $z(\eta_{T_i}-\mu_i\eta_{T_i})$ act as:
\begin{itemize}
    \item $\tau$, $\sigma+\sigma\tau$ and $z(\sigma-\sigma\tau)$ respectively, if $i=1$.
    \item $\sigma$, $\tau+\sigma\tau$ and $z(\tau-\sigma\tau)$ respectively, if $i=2$.
    \item $\tau\sigma$, $\sigma+\tau$ and $z(\sigma-\tau)$ respectively, if $i=3$.
\end{itemize}
In this procedure, we take  
$\sqrt{m}\sqrt{k}=\dfrac{m}{d}\sqrt{n},\ \sqrt{n}\sqrt{k}=\dfrac{n}{d}\sqrt{m},\ \sqrt{m}\sqrt{n}=d\sqrt{k}.$
Then, we apply the reduction method systematically with the three possible types of biquadratic extension listed in \ref{intbasisbiquadr}, which we call biquadratic extensions of first, second and third type henceforth. Note that, by \cite[Proposition 2.1]{truman}, the biquadratic extensions of third type are the tamely ramified ones.

\subsection{Biquadratic extensions of first type}

In this case, the integral basis $B$ has elements $\gamma_1=e_1,\ \gamma_2=e_2,\ \gamma_3=e_3,\ \gamma_4=\dfrac{e_3+e_4}{2},$ and the Gram matrix of the classical Galois structure is $$G(H_c,L_B)=\begin{pmatrix}
\gamma_1 & \gamma_2 & \gamma_3 & \gamma_4 \\
\gamma_1 & -\gamma_2 & \gamma_3 & \gamma_3-\gamma_4 \\
\gamma_1 & \gamma_2 & -\gamma_3 & -\gamma_4 \\
\gamma_1 & -\gamma_2 & -\gamma_3 & -\gamma_3+\gamma_4
\end{pmatrix}.$$ Then, for the non-classical Hopf-Galois structures, we have $$G(H_1,L)=\begin{pmatrix}
    \gamma_1 & \gamma_2 & \gamma_3 & \gamma_4 \\
    \gamma_1 & \gamma_2 & -\gamma_3 & -\gamma_4 \\
    2\gamma_1 & -2\gamma_2 & 0 & 0 \\
    0 & 0 & -2d\gamma_3+4d\gamma_4 & (-\frac{m}{d}-d)\gamma_3+2d\gamma_4
    \end{pmatrix},$$ 
    $$G(H_2,L)=\begin{pmatrix}
    \gamma_1 & \gamma_2 & \gamma_3 & \gamma_4 \\
    \gamma_1 & -\gamma_2 & \gamma_3 & \gamma_3-\gamma_4 \\
    2\gamma_1 & 0 & -2\gamma_3 & -\gamma_3 \\
    0 & -2d\gamma_3+4d\gamma_4 & 0 & -\frac{n}{d}\gamma_2
    \end{pmatrix},$$
    
    $$G(H_3,L)=\begin{pmatrix}
    \gamma_1 & \gamma_2 & \gamma_3 & \gamma_4 \\
    \gamma_1 & -\gamma_2 & -\gamma_3 & -\gamma_3+\gamma_4 \\
    2\gamma_1 & 0 & 0 & \gamma_3-2\gamma_4 \\
    0 & -\frac{2m}{d}\gamma_3 & \frac{2n}{d}\gamma_2 & \frac{n}{d}\gamma_2
    \end{pmatrix}.$$

\subsubsection{Reduced matrices}\label{sectredbiquad1}

The matrix $M(H_1,L)$ can be reduced to $$\begin{pmatrix}
1 & 1 & 2 & 0 \\
0 & 2 & 2 & 2d \\
0 & 0 & 4 & 0 \\
0 & 0 & 0 & d+\frac{m}{d} \\
0 & 0 & 0 & 4d
\end{pmatrix}.$$ With the last two rows 
we can use Bézout's identity
to get $g=\gcd(d+\frac{m}{d},4d)$ in one row and $0$ in the other one. We claim that $g=4$. Indeed, since $m\equiv3\pmod 4$ and $m=d\frac{m}{d}$, one of $d$ and $\frac{m}{d}$ is $3$ mod $4$ and the other one is $1$ mod $4$. Then,  $d+\frac{m}{d}\equiv 0\pmod 4$. On the other hand, since $m$ is squarefree, $\gcd(d+\frac{m}{d},d)=1$. 
Finally, we reduce the non-zero entry of the column to $2d\mod 4$, which is $2$, since $d$ is odd. 
Thus, $M(H_1,L)$ has Hermite normal form $$D(H_1,L)=\begin{pmatrix}
1 & 1 & 2 & 0 \\
0 & 2 & 2 & 2 \\
0 & 0 & 4 & 0 \\
0 & 0 & 0 & 4
\end{pmatrix}.$$ This leads to the index $I(H_1,L)=32$. \\ 

For $H_2$, we can reduce $M(H_2,L)$ to $$\begin{pmatrix}
1 & 0 & -1 & 0 \\
0 & 1 & -1 & 0 \\
0 & 0 & 4 & 0 \\
0 & 0 & 0 & \frac{n}{d} \\
0 & 0 & 0 & -2d
\end{pmatrix}.$$ 
Now we use $\gcd(2d,\frac{n}{d})=2$ since $n$ is even, $d$ is odd and $d$ and $\frac nd$ are coprime.
Then, the Hermite normal form is $$D(H_2,L)=\begin{pmatrix}
1 & 0 & -1 & 0 \\
0 & 1 & -1 & 0 \\
0 & 0 & 4 & 0 \\
0 & 0 & 0 & 2
\end{pmatrix}$$ and  the index is $I(H_2,L)=8$. \\ 

For the third non-classical Hopf-Galois structure, we may reduce $M(H_3,L)$ to the matrix $$\begin{pmatrix}
1 & 0 & -1 & 0 \\
0 & 1 & -1 & 0 \\
0 & 0 & 4 & 0 \\
0 & 0 & 0 & \frac{n}{d} \\
0 & 0 & 0 & 2\frac{m}{d}
\end{pmatrix}$$ 
and then use that $\gcd(\frac{n}{d},2\frac{m}{d})=2$ 
to get  Hermite normal form  $$D(H_3,L)=\begin{pmatrix}
1 & 0 & -1 & 0 \\
0 & 1 & -1 & 0 \\
0 & 0 & 4 & 0 \\
0 & 0 & 0 & 2
\end{pmatrix},$$ and $I(H_3,L)=8$.

\subsubsection{Freeness over the associated order}

Now, we study the freeness of $\mathcal{O}_L$ over its associated orders in $H_1$, $H_2$ and $H_3$. Given $\beta\in\mathcal{O}_L$, $$D_{\beta}(H_1,L)=-32\beta_1\beta_2\left(d\beta_3^2+d\beta_3\beta_4+\frac{1}{4}\left(d+\frac{m}{d}\right)\beta_4^2\right),$$ $$D_{\beta}(H_2,L)=8\beta_1(2\beta_3+\beta_4)\left(2d\beta_2^2+\frac{n}{2d}\beta_4^2\right),$$ $$D_{\beta}(H_3,L)=8\beta_1\beta_4\left(2\frac{m}{d}\beta_2^2+2\frac{n}{d}\beta_3^2+2\frac{n}{d}\beta_3\beta_4+\frac{n}{2d}\beta_4^2\right).$$ 

\begin{pro}\label{profreebiquadQ1} For $i\in\{1,2,3\}$, $\mathcal{O}_L$ is $\mathfrak{A}_{H_i}$-free if and only if there exist integers $x,y\in\mathbb{Z}$ such that at least one of the following equations is satisfied:
\begin{itemize}
    \item[1.] $x^2+my^2=\pm4d$, if $i=1$.
    \item[2.] $x^2+ny^2=\pm2d$, if $i=2$.
    \item[3.] $x^2+ky^2=\pm2\frac{n}{d}$, if $i=3$.
\end{itemize}
If that is the case, then a free generator of $\mathcal{O}_L$ as $\mathfrak{A}_{H_i}$-module is $$\beta=\begin{cases}
\gamma_1+\gamma_2+\frac{x-dy}{2d}\,\gamma_3+y\,\gamma_4 & \hbox{if }i=1 \\
\gamma_1+\frac{x}{2d}\,\gamma_2+\frac{1-y}{2}\,\gamma_3+y\,\gamma_4 & \hbox{if }i=2 \\
\gamma_1+\frac{y}{2}\,\gamma_2+\frac{xd-n}{2n}\,\gamma_3+\gamma_4 & \hbox{if }i=3
\end{cases}$$
\end{pro}
\begin{proof}
We proceed at each case as in Section \ref{sectcyclicquarticQ}.
\begin{itemize}
    \item[1.] $\mathcal{O}_L$ is $\mathfrak{A}_H$-free if and only if $$\beta_1\beta_2\left(d\beta_3^2+d\beta_3\beta_4+\frac{1}{4}\left(d+\frac{m}{d}\right)\beta_4^2\right)=s,$$ where $s\in\{-1,1\}$. We can always choose $\beta_1,\beta_2\in\mathbb{Z}$ such that $\beta_1\beta_2=\pm1$, so it is enough to consider the quadratic factor $$d\beta_3^2+d\beta_3\beta_4+\frac{1}{4}\left(d+\frac{m}{d}\right)\beta_4^2=s.$$ We regard this as a quadratic equation in $\beta_3$ with parameter $\beta_4=y$.
    The condition of the discriminant being a square gives 
    $-my^2+4ds=x^2$. Then, $\beta_3=\frac{-dy\pm x}{2d}$. The 
    equation is solvable in $\mathbb{Z}$ if and only if  $2d$ divides at least one of $-dy\pm x$. Let us see that this happens if and only if it divides both. Since $m\equiv -1\pmod 4$, we have $x^2\equiv y^2\pmod 4$ and therefore $x\equiv y\pmod 2$. Since $d$ is odd, this gives $-dy\pm x$ even. On the other hand, since $d$ divides $m$, we have $x^2\equiv 0\pmod d$. Since $d$ is square free, it divides $x$ and therefore $-dy\pm x$.
\item[2.] In this case, the equation we must consider is $$2d\beta_2^2+\dfrac{n}{2d}\beta_4^2=s.$$
    Namely, $4d^2\beta_2^2+{n}\beta_4^2=2ds$, whence $x=2d\beta_2$ and $y=\beta_4$. Note that, for any solution of $-ny^2+2ds=x^2$, since $2d$ divides $n$ and it is squarefree, it also divides $x$. On the other hand, the equation gives $2y^2\equiv 2\pmod 4$ and therefore $y$ odd.
    The remaining linear factor $2\beta_3+\beta_4$ becomes $1$ for $\beta_3=\frac{1-y}2$.
\item[3.] Here the situation is slightly different  since $$2\frac{n}{d}\beta_3^2+2\frac{n}{d}\beta_3\beta_4+2\frac{m}{d}\beta_2^2+\frac{n}{2d}\beta_4^2=s,$$ 
    is a ternary quadratic form. But $\beta_4$ is a factor of $D_{\beta}(H_3,L)$ and it must be $\pm 1$. We may assume that $\beta_4=1$ since it does not affect the discriminant, which is $4\left(-4k\beta_2^2+2\frac{n}{d}s\right).$ With $x,y$ as in the statement, $\beta_2=\dfrac{y}{2}$ and 
    $\beta_3=\dfrac{-2\frac{n}{d}\pm 2x}{4\frac{n}{d}}=\dfrac{-n\pm dx}{2n}.$ From  $x^2+ky^2=2\frac{n}{d}s$ and $n\equiv k\equiv 2\pmod 4$ we obtain that $x$ and $y$ are even. Equally, since $\frac nd$ divides $k$ and it is squarefree, we have that it divides $x$.  
    \end{itemize}
    \end{proof}

\begin{rmk}\normalfont
Note that the equations and generators may also be written as
\begin{itemize}
    \item[1.] $du^2+\dfrac md y^2=\pm4,\quad \beta=\gamma_1+\gamma_2+\dfrac{u-y}{2}\,\gamma_3+y\,\gamma_4$, 
    \item[2.] $du^2+\dfrac nd y^2=\pm2,\quad \beta=\gamma_1+\dfrac{u}{2}\,\gamma_2+\dfrac{1-y}{2}\,\gamma_3+y\,\gamma_4$,
    \item[3.] $\dfrac nd u^2+\dfrac mdy^2=\pm2,\quad \beta=\gamma_1+\dfrac{y}{2}\,\gamma_2+\dfrac{u-1}2\,\gamma_3+\gamma_4 $,
\end{itemize}
by replacing $x=d u$ in cases 1 and 2, and $x=\dfrac nd u$ in case 3. 
\end{rmk}

\begin{rmk}\normalfont
If $m,n$ are positive and coprime, $\mathcal{O}_L$ is always $\mathfrak{A}_{H_1}$-free with generator $\gamma_1+\gamma_2+\gamma_3$. For $i\in\{2,3\}$, $\mathcal{O}_L$ is $\mathfrak{A}_{H_i}$-free if and only if $n=2$. In that case, $\gamma_1+\gamma_4$ is a generator. 
\end{rmk}

\subsection{Biquadratic extensions of second type}

We have $$\gamma_1=e_1,\quad \gamma_2=\frac{e_1+e_2}{2},\quad\gamma_3=e_3,\quad\gamma_4=\frac{e_3+e_4}{2},$$ and then $$G(H_c,L_B)=\begin{pmatrix}
\gamma_1 & \gamma_2 & \gamma_3 & \gamma_4 \\
\gamma_1 & \gamma_1-\gamma_2 & \gamma_3 & \gamma_3-\gamma_4 \\
\gamma_1 & \gamma_2 & -\gamma_3 & -\gamma_4 \\
\gamma_1 & \gamma_1-\gamma_2 & -\gamma_3 & -\gamma_3+\gamma_4
\end{pmatrix}.$$ Thus, for the non-classical Hopf-Galois structures, one computes $$G(H_1,L)=\begin{pmatrix}
    \gamma_1 & \gamma_2 & \gamma_3 & \gamma_4 \\
    \gamma_1 & \gamma_2 & -\gamma_3 & 
    -\gamma_4 \\
    2\gamma_1 & 2\gamma_1-2\gamma_2 & 0 & 0 \\
    0 & 0 & -2d\gamma_3+4d\gamma_4 & (-\frac{m}{d}-d)\gamma_3+2d\gamma_4
    \end{pmatrix},$$ $$G(H_2,L)=\begin{pmatrix}
    \gamma_1 & \gamma_2 & \gamma_3 & \gamma_4 \\
    \gamma_1 & \gamma_1-\gamma_2 & \gamma_3 & \gamma_3-\gamma_4 \\
    2\gamma_1 & \gamma_1 & -2\gamma_3 & -\gamma_3 \\
    0 & -d\gamma_3+2d\gamma_4 & 0 & \frac{n}{d}\gamma_1-\frac{2n}{d}\gamma_2
    \end{pmatrix},$$ $$G(H_3,L)=\begin{pmatrix}
    \gamma_1 & \gamma_2 & \gamma_3 & \gamma_4 \\
    \gamma_1 & \gamma_1-\gamma_2 & -\gamma_3 & -\gamma_3+\gamma_4 \\
    2\gamma_1 & \gamma_1 & 0 & \gamma_3-2\gamma_4 \\
    0 & -\frac{m}{d}\gamma_3 & -\frac{2n}{d}\gamma_1+\frac{4n}{d}\gamma_2 & -\frac{n}{d}\gamma_1+\frac{2n}{d}\gamma_2
    \end{pmatrix}.$$
    
\subsubsection{Reduced matrices}\label{sectredbiquad2}

For $H_1$, we may reduce the matrix of the action to $$\begin{pmatrix}
1 & 1 & 0 & 0 \\
0 & 2 & 0 & 2d \\
0 & 0 & 2 & 0 \\
0 & 0 & 0 & \frac{m}{d}+d \\
0 & 0 & 0 & 4d
\end{pmatrix}.$$ 
We have $\gcd(\frac{m}{d}+d,4d)=2$ since
$d$ and $\frac nd$ are coprime and $m=d\frac{m}{d}\equiv1\pmod 4$, so $d\equiv\frac{m}{d}\pmod 4$, and then $d+\frac{m}{d}\equiv2\pmod 4$. 
Therefore, the Hermite normal form in this case is $$D(H_1,L)=\begin{pmatrix}
1 & 1 & 0 & 0 \\
0 & 2 & 0 & 0 \\
0 & 0 & 2 & 0 \\
0 & 0 & 0 & 2
\end{pmatrix},$$ and $I(H_1,L)=8$. 
We reduce $M(H_2,L)$ and $M(H_3,L)$ to 
$$\begin{pmatrix}
1 & 0 & 1 & 0 \\
0 & 1 & 1 & 0 \\
0 & 0 & 2 & 0 \\
0 & 0 & 0 & \frac{n}{d} \\
0 & 0 & 0 & d
\end{pmatrix},
\quad
\begin{pmatrix}
1 & 0 & 1 & 0 \\
0 & 1 & 1 & 0 \\
0 & 0 & 2 & 0 \\
0 & 0 & 0 & \frac{m}{d} \\
0 & 0 & 0 & \frac{n}{d}
\end{pmatrix}
$$ respectively.  Since $d$ and $\frac{n}{d}$ are coprime, and the same happens for 
$\frac{m}{d}$ and $\frac{n}{d}$, in both cases we get
Hermite normal form $$D(H_i,L)=\begin{pmatrix}
1 & 0 & 1 & 0 \\
0 & 1 & 1 & 0 \\
0 & 0 & 2 & 0 \\
0 & 0 & 0 & 1
\end{pmatrix}$$
and index $I(H_i,L)=2$. 

\subsubsection{Freeness over the associated order}

Let us study the freeness. Given $\beta\in\mathcal{O}_L$,
$$\begin{array}{l}
D_{\beta}(H_1,L)=-8\beta_2(2\beta_1+\beta_2)\left(2d\beta_3^2+2d\beta_3\beta_4+\frac{1}{2}\left(d+\frac{m}{d}\right)\beta_4^2\right),\\[1ex]
D_{\beta}(H_2,L)=4(2\beta_1+\beta_2)(2\beta_3+\beta_4)\left(d\beta_2^2+\frac{n}{d}\beta_4^2\right),
\\[1ex] D_{\beta}(H_3,L)=4\beta_4(2\beta_1+\beta_2)\left(\frac{m}{d}\beta_2^2+4\frac{n}{d}\beta_3^2+4\frac{n}{d}\beta_3\beta_4+\frac{n}{d}\beta_4^2\right).
\end{array}$$ We see that $4$ divides $D_{\beta}(H_i,L)$ for $i\in\{2,3\}$ while $I(H_i,L)=2$ for $i\in\{1,2\}$. Hence, $\mathcal{O}_L$ is neither $\mathfrak{A}_{H_2}$-free nor $\mathfrak{A}_{H_3}$-free. We are left with $H_1$.

\begin{pro}\label{profreebiquadQ2} $\mathcal{O}_L$ is $\mathfrak{A}_{H_1}$-free if and only if there exist integers $x,y\in\mathbb{Z}$ such that at least one of the equations $x^2+my^2=\pm2d$ is satisfied. If it is so, a free generator of $\mathcal{O}_L$ as $\mathfrak{A}_{H_1}$-module is $$\beta=\gamma_1-\gamma_2+\frac{x-dy}{2d}\,\gamma_3+y\,\gamma_4.$$
\end{pro}
\begin{proof}
For $s\in\{-1,1\}$, the equation to consider is $$2d\beta_3^2+2d\beta_3\beta_4+\frac{1}{2}\left(d+\frac{m}{d}\right)\beta_4^2=s.$$ 
Taking $\beta_4=y$,
 the discriminant of the equation is $4(-my^2+2ds)$, so this being a square is equivalent to the existence of $x$ and $y$ as in the statement. Then, $\beta_3=\frac{-dy\pm x}{2d}$. 
 From $x^2+my^2=2ds$ we deduce that $d$ divides $x$ and also that $x^2+y^2\equiv2\pmod 4$, which implies that both are odd.
Since $d$ is odd, $-dy\pm x$ is even.

 \begin{rmk}\normalfont
Note that by replancing $x=d u$ the equation and generator may also be written as
$$du^2+\dfrac md y^2=\pm2,\quad \beta=
\gamma_1-\gamma_2+\frac{u-y}{2}\,\gamma_3+y\,\gamma_4.
$$
    \end{rmk}
\end{proof}

\subsection{Biquadratic extensions of third type: the tame case}

The integral basis $B$ is formed by $$\gamma_1=e_1,\quad \gamma_2=\frac{e_1+e_2}{2},\quad\gamma_3=\frac{e_1+e_3}{2},\quad\gamma_4=\frac{1}{4}e_1+\frac{1}{4}e_2+\frac{m}{4d}e_3+\frac{1}{4}e_4,$$ and then 
$$G(H_c,L_B)=\begin{pmatrix}
\gamma_1 & \gamma_2 & \gamma_3 & \gamma_4 \\
\gamma_1 & \gamma_1-\gamma_2 & \gamma_3 & \frac{d-m}{2d}\gamma_1+\frac{m}{d}\gamma_3-\gamma_4 \\
\gamma_1 & \gamma_2 & \gamma_1-\gamma_3 & \gamma_2-\gamma_4 \\
\gamma_1 & \gamma_1-\gamma_2 & \gamma_1-\gamma_3 & \frac{d+m}{2d}\gamma_1-\gamma_2-\frac{m}{d}\gamma_3+\gamma_4
\end{pmatrix}.$$ Then, the Gram matrices of the non-classical Hopf-Galois structures are $$G(H_1,L)=\begin{pmatrix}
    \gamma_1 & \gamma_2 & \gamma_3 & \gamma_4 \\
    \gamma_1 & \gamma_2 & \gamma_1-\gamma_3 & 
    \gamma_2-\gamma_4 \\
    2\gamma_1 & 2\gamma_1-2\gamma_2 & \gamma_1 & \gamma_1-\gamma_2 \\
    0 & 0 & x & y
    \end{pmatrix},$$ $$G(H_2,L)=\begin{pmatrix}
    \gamma_1 & \gamma_2 & \gamma_3 & \gamma_4 \\
    \gamma_1 & \gamma_1-\gamma_2 & \gamma_3 & \left(\frac{1}{2}-\frac{m}{2d}\right)\gamma_1+\frac{m}{d}\gamma_3-\gamma_4 \\
    2\gamma_1 & \gamma_1 & 2\gamma_1-2\gamma_3 & \frac{(m+d)\gamma_1-2m\gamma_3}{2d} \\
    0 & z & 0 & t
    \end{pmatrix},$$ $$G(H_3,L)=\begin{pmatrix}
    \gamma_1 & \gamma_2 & \gamma_3 & \gamma_4 \\
    \gamma_1 & \gamma_1-\gamma_2 & \gamma_1-\gamma_3 & \frac{d+m}{2d}\gamma_1-\gamma_2-\frac{m}{d}\gamma_3+\gamma_4 \\
    2\gamma_1 & \gamma_1 & \gamma_1 & \frac{d-m}{2d}\gamma_1+\gamma_2+\frac{m}{d}\gamma_3-2\gamma_4 \\
    0 & \frac{m}{d}\gamma_1-2\frac{m}{d}\gamma_3 & -\frac{n}{d}\gamma_1+\frac{2n}{d}\gamma_2 & \left(-\frac{m}{2d^2}+\frac{m}{2d}\right)\gamma_1+\frac{m}{d^2}\gamma_2-\frac{m}{d}\gamma_3
    \end{pmatrix},
    $$ 
    where 
    $$x=m\gamma_1-2d\gamma_2-2m\gamma_3+4d\gamma_4,\quad y=\frac{m(m+1)}{2d}\gamma_1-m\gamma_2-\frac{m(m+1)}{d}\gamma_3+2m\gamma_4,$$ $$z=m\gamma_1-2d\gamma_2-2m\gamma_3+4d\gamma_4,\quad t=\left(\frac{n}{2d}+\frac{m}{2}\right)\gamma_1-\left(d+\frac{n}{d}\right)\gamma_2-m\gamma_3+2d\gamma_4.$$
    
\subsubsection{Reduced matrices}\label{sectredbiquad3}
 
The matrix of the action $M(H_1,L)$ reduces to $$\begin{pmatrix}
1 & 0 & 0 & m\left(\frac{m+1}{2d}-1\right) \\
0 & 1 & 0 & -m\left(\frac{m+1}{2d}-1\right) \\
0 & 0 & 1 & -\frac{m(m+1)}{2d} \\
0 & 0 & 0 & \frac{m}{d}(m+1) \\
0 & 0 & 0 & 2d
\end{pmatrix}$$ 
and we have to compute $\gcd\left(\frac{m}{d}(m+1),2d\right)$.
Since $\frac md$ is odd and coprime with $d$, it is equal to  
$\gcd\left(m+1,2d\right)$, which is $2$ since $m\equiv 1\pmod 4$ and $d$ divides $m$. 

Then, the matrix above is equivalent to $$\begin{pmatrix}
1 & 0 & 0 & m\left(\frac{m+1}{2d}-1\right) \\
0 & 1 & 0 & -m\left(\frac{m+1}{2d}-1\right) \\
0 & 0 & 1 & -\frac{m(m+1)}{2d} \\
0 & 0 & 0 & 2
\end{pmatrix}.$$ The entries above $2$ in the fourth column reduce to $0$ or $1$ depending on their parity. Therefore, the Hermite normal form of $M(H_1,L)$ is: $$D(H_1,L)=\begin{pmatrix}
1 & 0 & 0 & 0 \\
0 & 1 & 0 & 0 \\
0 & 0 & 1 & 1 \\
0 & 0 & 0 & 2
\end{pmatrix}.$$ Then, $I(H_1,L)=2$. 

For the second one, we reduce the matrix of the action to $$\begin{pmatrix}
1 & 0 & 0 & \frac{3m^2+n}{2d} \\
0 & 1 & 0 & \frac{-9m^2+n}{2d} \\
0 & 0 & 1 & \frac{-2md-3m^2-n}{2d} \\
0 & 0 & 0 & d+\frac{n}{d} \\
0 & 0 & 0 & m+d \\
0 & 0 & 0 & \frac{2n}{d} \\
0 & 0 & 0 & 2d
\end{pmatrix}.$$ Since the greatest common divisor of $2d$ and $\frac{2n}{d}$ is $2$, arguing as in previous cases, we obtain the same Hermite normal form $D(H_2,L)=D(H_1,L)$ and therefore
also $I(H_2,L)=2$. 

Finally, we reduce $M(H_3,L)$ to $$\begin{pmatrix}
1 & 0 & 0 & \frac{nd^2+m(m-n)d-m^2n}{2d^3} \\
0 & 1 & 0 & \frac{nd^2+m(m-n)d-m^2n}{2d^3} \\
0 & 0 & 1 & \frac{nd^2-m(m-n)d+m^2n}{2d^3} \\
0 & 0 & 0 & \frac{m^2}{d^2}(d-n) \\
0 & 0 & 0 & \frac{m+d}{d}\frac{m}{d} \\
0 & 0 & 0 & \frac{m+n}{d} \\
0 & 0 & 0 & \frac{2n}{d}
\end{pmatrix}.$$ Let us focus in the last two entries. Since $m$ and $n$ are $1$ mod $4$, $m+n$ is $2$ mod $4$, just as $2n$. Then, $\mathrm{gcd}(m+n,2n)=2\mathrm{gcd}(m+n,n)=2d$. Thus, $\mathrm{gcd}(\frac{m+n}{d},\frac{2n}{d})=2$. Then, the Hermite normal form and the index are exactly the same as before.

\subsubsection{Freeness over the associated order}

Let $\beta\in\mathcal{O}_L$. Then, we have $$
\begin{array}{l}
D_{\beta}(H_1,L)=-2(2\beta_2+\beta_4)(4\beta_1+2\beta_2+2\beta_3+\beta_4)\left(2d\beta_3^2+2m\beta_3\beta_4+\frac{m}{d}\frac{m+1}{2}\beta_4^2\right),
\\[1ex]
D_{\beta}(H_2,L)=2\left(2\beta_3+\frac{m}{d}\beta_4\right)(4\beta_1+2\beta_2+2\beta_3+\beta_4)\left(2d\beta_2^2+2d\beta_2\beta_4+\frac{1}{2}\left(d+\frac{n}{d}\right)\beta_4^2\right),$$ \\[1ex]
D_{\beta}(H_3,L)=2\beta_4(4\beta_1+2\beta_2+2\beta_3+\beta_4)q_3(\beta_1,\beta_2,\beta_3,\beta_4),
\end{array}$$ where
$$
        q_3(\beta_1,\beta_2,\beta_3,\beta_4)=2\frac{m}{d}\beta_2^2+2\frac{m}{d}\beta_2\beta_4+2\frac{n}{d}\beta_3^2+2k\beta_3\beta_4+\frac{m}{d}\frac{k+1}{2}\beta_4^2.
$$

\begin{pro}\label{profreebiquadQ3} For $i\in\{1,2,3\}$, $\mathcal{O}_L$ is $\mathfrak{A}_{H_i}$-free if and only if there exist integers $x,y\in\mathbb{Z}$ such that at least one of the following equations is satisfied:
\begin{itemize}
    \item[1.] $x^2+my^2=\pm 2d$, if $i=1$.
    \item[2.] $x^2+ny^2=\pm 2d$, if $i=2$.
    \item[3.] $x^2+ky^2=\pm 2\frac{n}{d}$, if $i=3$.
\end{itemize}

In that case, a generator of $\mathcal{O}_L$ as $\mathfrak{A}_{H_i}$-module is $$\beta=\begin{cases}
\dfrac12\left(\dfrac{my-x}{2d}-\epsilon\right)\,\gamma_1+\dfrac{1-y}{2}\,\gamma_2+\dfrac{x-my}{2d}\,\gamma_3+y\,\gamma_4 & \hbox{if }i=1
,\ \epsilon=
\frac{x-my}{2d}\pmod 2 ;\\[2ex]
\dfrac14\left(\dfrac{my-x}{d}-\epsilon\right)\,\gamma_1+\dfrac{x-yd}{2d}\,\gamma_2+
\dfrac{d-my}{2d}\,\gamma_3+y\,\gamma_4 & \hbox{if }i=2
,\ \epsilon=\frac{my-x}{d}\pmod 4;\\[3ex]
\dfrac{1}{4}\left(\dfrac{k-x}{\frac nd}-y+\epsilon\right)\,\gamma_1+\dfrac{y-1}{2}\,\gamma_2+\dfrac{x-k}{2\frac{n}{d}}\,\gamma_3+\gamma_4 & \hbox{if }i=3,\  \epsilon=(\frac{k-x}{\frac{n}{d}}-y)\pmod 4.
\end{cases}$$
\end{pro}
\begin{proof}
First, note that if $x$ and $y$ are integers that satisfy any of the equalities above, then $x$ and $y$ are necessarily odd,
since $x^2+y^2\equiv2\pmod 4$.

\begin{itemize}
    \item[1.] We proceed as in the previous cases. The equation to consider is $$2d\beta_3^2+2m\beta_3\beta_4+\frac{m}{d}\frac{m+1}{2}\beta_4^2=s,$$ 
    with discriminant $4(-m\beta_4^2+2ds)$. 
    Therefore, $\beta_4=y$ and 
    $\beta_3=\frac{-my\pm x}{2d}$. Both values are integer because $d$ divides both $x$ and $m$, and $-my+x$ is even. Once we have found $\beta_3$ and $\beta_4$ representing $\pm 1$ we need the linear factors to be also $\pm 1$. 
    Solving $2\beta_2+\beta_4=1$ and  $4\beta_1+2\beta_2+2\beta_3+\beta_4=1$, we get 
    $\beta_2=\dfrac{1-\beta_4}2=\dfrac{1-y}2$
    and $\beta_1=-\dfrac{\beta_3}2$. It is integer when $\beta_3$ is even. But solving $2\beta_2+\beta_4=1$ and  $4\beta_1+2\beta_2+2\beta_3+\beta_4=-1$, we get 
    $\beta_2=\dfrac{1-y}2$
    and $\beta_1=\dfrac{-1-\beta_3}2$, which is integer when $\beta_3$ is odd.  
    \item[2.] In the second case, the quadratic equation is $$2d\beta_2^2+2d\beta_2\beta_4+\frac{1}{2}\left(d+\frac{n}{d}\right)\beta_4^2=s$$
    and the discriminant is $4(-ny^2+2ds)$.  We get solutions $\beta_4=y$ and $\beta_2=\frac{-dy\pm x}{2d}$. 
    Once we have $\beta_2$ and $\beta_4$, we determine the linear factors to be unities.Taking $\beta_3=\frac 12(1-\frac md y)=\frac{d-my}{2d}$ the first one becomes $1$. The equations $4\beta_1+2\beta_2+2\beta_3+\beta_4=\pm1$
    lead to $4\beta_1=\frac{my-x}{d}-\epsilon$ with $\epsilon$ being $0$ or $2$. Since $\frac{my-x}{d}$ is even, we always have a choice which makes the second term congruent to $0\pmod 4$.
     \item[3.] Finally, we must consider the equation $$2\frac{m}{d}\beta_2^2+2\frac{m}{d}\beta_2\beta_4+2\frac{n}{d}\beta_3^2+2k\beta_3\beta_4+\frac{m}{d}\frac{k+1}{2}\beta_4^2=s,$$ 
    and we choose unknown $\beta_3$ and parameters $\beta_2$ and $\beta_4$, with $\beta_4=\pm1$ since it is a factor of $D_{\beta}(H_3,L)$.  The discriminant is $4((2\beta_2+\beta_4)^2k+2\frac{n}{d}s)$. We take $2\beta_2+\beta_4=y$ and then $\beta_3=\frac{-k\beta_4\pm x}{2\frac{n}{d}}$, which are integers since $\frac nd$ divides both $k$ and $x$  and the numerator is even.
    We can safely take $\beta_4=1$, $\beta_2=\frac{y-1}2$,
    $\beta_3=\frac{x-k}{2\frac{n}{d}}$. The remaining condition $4\beta_1+2\beta_2+2\beta_3+\beta_4=\pm1$
    leads to $4\beta_1=\pm1 -y+\frac{k-x}{\frac nd}$.
    Since the second summand is odd, there is always a choice of sign which makes the right term divisible by $4$.
\end{itemize}
\end{proof}

\begin{rmk}\normalfont The criteria obtained in Proposition \ref{profreebiquadQ3} were proved by Truman in \cite[Proposition 6.1]{truman} using the theory of idèles. In his case, he works indistinctly with a non-classical Hopf-Galois structure of a tame biquadratic extension $\mathbb{Q}(\sqrt{m},\sqrt{n})/\mathbb{Q}$ and obtains the same condition in terms of the chosen numbers $m$ and $n$. This fits with our result because $m,n,k$ being $1$ mod $4$ allows to exchange them indistinctly, and $\frac{n}{d}=\mathrm{gcd}(n,k)$. Our Propositions \ref{profreebiquadQ1} and \ref{profreebiquadQ2} show that $\mathcal{O}_L$ presents a similar behaviour as $\mathfrak{A}_{H_i}$-module in case $L/\mathbb{Q}$ is wildly ramified.
\end{rmk}

\subsection{Summary of results}

\begin{teo} Let $L/\mathbb{Q}$ be a biquadratic extension of number fields and let $H_1,\,H_2,\,H_3$ be its non-classical Hopf-Galois structures. The following table summarizes the freeness of $\mathcal{O}_L$ as $\mathfrak{A}_{H_i}$-module for $i\in\{1,2,3\}$.

\normalfont
\begin{center}
\begin{tabular}{|c|c|c|c|c|} \hline
    \multicolumn{2}{|c|}{Mod $4$} & \multicolumn{3}{|c|}{$\mathcal{O}_L$ as $\mathfrak{A}_{H_i}$-module} \\ \hline
    $m$ & $n$ & $H_1$ & $H_2$ & $H_3$ \\ \hline
    $1$ & $1$ & Free $\Longleftrightarrow$ & Free $\Longleftrightarrow$ & Free $\Longleftrightarrow$ \\
    & & $\exists x,y\in\mathbb{Z}\,\colon$ & $\exists x,y\in\mathbb{Z}\,\colon$ & $\exists x,y\in\mathbb{Z}\,\colon$\\
    & & $x^2+my^2=\pm2d$ & $x^2+ny^2=\pm2d$ & $x^2+ky^2=\pm2\frac{n}{d}$ \\ \hline
    $1$ & $\neq1$ & Free $\Longleftrightarrow$ & Not free & Not free\\
    & & $\exists x,y\in\mathbb{Z}\,\colon$ & & \\
    & & $x^2+my^2=\pm2d$ & & \\ \hline
    $3$ & $2$ & Free $\Longleftrightarrow$ & Free $\Longleftrightarrow$ & Free $\Longleftrightarrow$ \\
    & & $\exists x,y\in\mathbb{Z}\,\colon$ & $\exists x,y\in\mathbb{Z}\,\colon$ & $\exists x,y\in\mathbb{Z}\,\colon$\\
    & & $x^2+my^2=\pm4d$ & $x^2+ny^2=\pm2d$ & $x^2+ky^2=\pm2\frac{n}{d}$ \\ \hline
\end{tabular}
\end{center}
\end{teo}

\section{The connection with generalized Pell equations}\label{sectgenpell}

The conditions obtained in Theorem \ref{teofreenesscyclic} and Propositions \ref{profreebiquadQ1}, \ref{profreebiquadQ2} and \ref{profreebiquadQ3} refer to the solvability in $\mathbb{Z}$ of equations of the form $x^2-Dy^2=N$. This is known as the generalized Pell equation or the Pell-Fermat equation. The equations of this type have been widely studied and algorithms of resolution have been developed (see for example \cite[Section 6.3.5]{cohen}). In \cite[Section 7]{truman}, Truman takes advantage from this fact to obtain results on the behaviour of $\mathcal{O}_L$ when $L/\mathbb{Q}$ is tame biquadratic. Namely, he gives examples of tame biquadratic extensions $L/\mathbb{Q}$ such that $\mathcal{O}_L$ is free over the associated order in $0$, $1$ or $2$ non-classical Hopf-Galois structures, and proves that there is no one with freeness in the three of them.

In this section, we aim to translate the theory of generalized Pell equations to Hopf-Galois theory so as to find results for any quartic Galois extension of $\mathbb{Q}$. Moreover, in the biquadratic cases, Hopf-Galois structures can be exchanged to obtain results on the solvability of the Pell equations themselves. 

\subsection{Cyclic quartic Galois extensions of $\mathbb Q$}

Let us take a cyclic quartic extension $L/\mathbb{Q}$ defined by $a$, $b$, $c$ and $d$. In this section, we follow the convention fixed in Section \ref{sectfreenesscyclic}, i.e. that $b$ is odd, and then $\mathcal{O}_L$ is $\mathfrak{A}_H$-free if and only if at least one of the equations $x^2-dy^2=\pm b$ is solvable and $b$ divides $x-cy$ for some solution $(x,y)$ (see Theorem \ref{teofreenesscyclic}).

When the independent term is $1$, we recover a classical Pell equation, which has the trivial solution $(x,y)=(1,0)$.

\begin{coro} If $b=1$, then $\mathcal{O}_L$ is $\mathfrak{A}_H$-free.
\end{coro}

Another quick fact is inspired by Example \ref{examplecyclic}, where for the second extension we used that $3$ is not a square mod $10$ to conclude non-freeness. 

\begin{pro} Let $L=\mathbb{Q}(\sqrt{a(d+b\sqrt{d})})$ be a cyclic quartic field.
If $d$ is odd and we have Jacobi symbol $\genfrac(){}{0}{b}{d}=-1$ or $d$ is even and
$\genfrac(){}{0}{b}{d/2}=-1$
then $\mathcal{O}_L$ is not $\mathfrak{A}_H$-free.
\end{pro}
\begin{proof}
Under the hypothesis, there is an odd prime $p\mid d$ such that $\genfrac(){}{0}{b}{p}=-1$.
Then $b$ is not a quadratic residue mod $p$ and $x^2-dy^2=b$ is not solvable in $\mathbb{Z}$, 
so $\mathcal{O}_L$ is not $\mathfrak{A}_H$-free.
\end{proof}

Additionally, results on the non solvability of certain Pell equation can now be read as results on the non freeness of $\mathcal{O}_L$ over the associated order. 

\begin{pro} Let $N=m^2n$ with $n\in\mathbb{Z}$ square-free. If $p$ is prime and $p\equiv n\equiv3\pmod 4$, the equation $x^2-py^2=N$ is not solvable.
\end{pro}
\begin{proof}
See \cite[Theorem 4.2.5]{andreescuandrica}.
\end{proof}

\begin{coro} Assume that $b=m^2n$ with $n\in\mathbb{Z}$ square-free. If $d$ is prime  and $d\equiv n\equiv 3\pmod 4$, then $x^2-dy^2=b$ is not solvable and $\mathcal{O}_L$ is not $\mathfrak{A}_H$-free.
\end{coro}

\begin{teo} If $p$ is prime and $\genfrac(){}{0}{p}{q}=-1$ for some odd prime divisor $q$ of $N$, then $x^2-py^2=N$ is not solvable.  
\end{teo}
\begin{proof}
See \cite[Theorem 4.2.8 and Corollary 4.2.9]{andreescuandrica}.
\end{proof}

\begin{coro} Assume that $d$ is prime. If $\genfrac(){}{0}{d}{q}=-1$ for some odd prime divisor $q$ of $b$, then $x^2-dy^2=b$ is not solvable and $\mathcal{O}_L$ is not $\mathfrak{A}_H$-free.
\end{coro}

\subsection{Wildly ramified biquadratic extensions of first type}

We continue by exploring the family of wildly ramified biquadratic extensions $L/\mathbb{Q}$ of first type, that is $L=\mathbb{Q}(\sqrt{m},\sqrt{n})$ with $m\equiv\,3\pmod 4$ and $n\equiv2\pmod 4$. Recall that
for $H_1,\ H_2, H_3$, freeness of $\mathcal{O}_L$ over the associated order is given by the existence of integer solutions of the equations
$$
     du^2+\dfrac mdy^2=\pm4,\quad
    du^2+\dfrac ndy^2=\pm2,\quad 
  \dfrac nd u^2+\dfrac mdy^2=\pm2,
  $$
  respectively. Let us study what happens when $m$, $n$ or $k$ are positive.

\begin{pro}\label{biquadQ1posit} Let $L=\mathbb{Q}(\sqrt{m},\sqrt{n})$ be a biquadratic extension of $\mathbb{Q}$ with $m\equiv3\pmod 4$ and $n\equiv2\pmod 4$. Let $d=\mathrm{gcd}(m,n)$ and $k=\frac{mn}{d^2}$.
\begin{itemize}
    \item[1.] If $m>0$, $\mathcal{O}_L$ is not $\mathfrak{A}_{H_1}$-free unless $m$ and $n$ are coprime or $m$ divides $n$.
    \item[2.] If $n>0$, then $\mathcal{O}_L$ is not $\mathfrak{A}_{H_2}$-free  unless $n=2d$. 
    \item[3.] If $k>0$, then $\mathcal{O}_L$ is not $\mathfrak{A}_{H_3}$-free  unless $|n|=2d$. 
\end{itemize}
\end{pro}
\begin{proof}
\begin{itemize}
    \item[1.] Since $m$ and $d$ are positive, we are looking for integer solutions of $du^2+\dfrac mdy^2= 4$.
    The only possibilities are $0+4$, which implies $\dfrac md=1$; $1+3$, which implies $d=1, \ m=3$, and $4+0$ wihch implies $d=1$.  Since $m=d=\gcd(m,n)$ implies $m$ divides $n$, we are done. 
    \item[2.] Since $n$ and $d$ are positive, we are looking for integer solutions of $du^2+\dfrac ndy^2= 2$.
    Since $d$ is odd and $n\ne 1$, the only possibility is $0+2$, which implies $\dfrac nd=2$.  
    
    \item[3.] We are looking for integer solutions of $\dfrac nd u^2+\dfrac mdy^2=\pm 2$ where the coeffiecients of the left hand side have the same sign. 
    Since it is not possible $m=n=d$ or $m=n=-d$, and $m$ is odd, the only possibilities are $2+0$ or $-2+0$, namely $n=\pm 2d$.
 \end{itemize}
\end{proof}
\begin{rmk}\normalfont Actually, the two exceptions of non-freeness in $H_1$ are essentially the same. Indeed, $m$ and $n$ are coprime if and only if $m$ divides $k$, and we have also the corresponding equivalence obtained by exchanging $n$ and $k$.
\end{rmk}

If we impose that $m$, $n$ and $k$ are positive, we should obtain an extension $L/\mathbb{Q}$ such that $\mathcal{O}_L$ is free over the associated order in all Hopf-Galois structures, whenever the exceptions obtained in Proposition \ref{biquadQ1posit} are compatible. We see that this happens exactly when $n=2$ or $k=2$.

\begin{coro}\label{uniquetotrealcase1} The unique totally real biquadratic extensions $L=\mathbb{Q}(\sqrt{m},\sqrt{n})$ of $\mathbb{Q}$ with $m\equiv3\pmod 4$ and $n\equiv2\pmod 4$ for which $\mathcal{O}_L$ is $\mathfrak{A}_{H_i}$-free for all $i\in\{1,2,3\}$ are those of the form $L=\mathbb{Q}(\sqrt{m},\sqrt{2})$.
\end{coro}
\begin{proof}
Since $L/\mathbb{Q}$ is totally real, $m,n,k>0$. By Proposition \ref{biquadQ1posit}, $\mathcal{O}_L$ is $\mathfrak{A}_{H_1}$-free only if $m$ and $n$ are coprime or $m$ divides $n$. Now, $\mathcal{O}_L$ is $\mathfrak{A}_{H_2}$-free and $\mathfrak{A}_{H_3}$-free only for $n=2d$, which in the first case gives $n=2$. 
But $m$ odd dividing $2d$ gives $m=d$, and then $k=2$.
\end{proof}

Proposition \ref{biquadQ1posit} shows that freeness in a non-classical Hopf-Galois structure is not common for totally real extensions, in the sense that it does not hold with some exceptions. In fact, we can see that under these exceptions, there is always freeness, regardless of the sign of $m$, $n$ or $k$.

\begin{coro}\label{corofreenesswildsep}
\begin{itemize}
    \item[1.] If $m$ and $n$ are coprime or $m$ divides $n$, then $\mathcal{O}_L$ is $\mathfrak{A}_{H_1}$-free.
    \item[2.] If $n=\pm2d$, then $\mathcal{O}_L$ is $\mathfrak{A}_{H_2}$-free and $\mathfrak{A}_{H_3}$-free.
\end{itemize}
\end{coro}
\begin{proof}
\begin{itemize}
    \item[1.] If $m$ and $n$ are coprime, the first equation becomes $u^2+my^2=4$, which has solution $(u,y)=(2,0)$. If $m$ divides $n$, then $m=d$ and the equation becomes $du^2+y^2=4$, which has solution $(u,y)=(0,2)$.
    \item[2.] If $n=\pm2d$, then the equation $du^2\pm 2y^2=\pm2$ is satisfied for $u=0$ and $y=1$. Moreover, the third equation becomes  $\pm2 x^2+\frac md y^2=\pm2$, which is satisfied for $x=1$ and $y=0$.
\end{itemize}
\end{proof}

In particular, if $|m|$ is prime, then $m$ and $n$ are coprime or $m$ divides $n$, so $\mathcal{O}_L$ is $\mathfrak{A}_{H_1}$-free.

We can think again of extensions $L/\mathbb{Q}$ that satisfy all the conditions in Corollary \ref{corofreenesswildsep} simultaneously, so that $\mathcal{O}_L$ is $\mathfrak{A}_H$-free for all non-classical Hopf-Galois structures $H$. Since we impose that $m$ is coprime with one of $n$ or $k$ (so it divides the other one) and $n=\pm2d$, necessarily $n=\pm2$ or $k=\pm2$. Then, we have:

\begin{coro}\label{sufffreenesswild} If $n=\pm2$ or $k=\pm2$, then $\mathcal{O}_L$ is $\mathfrak{A}_H$-free for every Hopf-Galois structure $H$ of $L/\mathbb{Q}$.
\end{coro}
\begin{proof}
Since $n$ and $k$ are exchangeable, it is enough to prove the statement for $n=\pm2$. Then, $m$ and $n$ are coprime, so $\mathcal{O}_L$ is $\mathfrak{A}_{H_1}$-free. Moreover, this also means that $n=\pm2d$ as $d=1$, so $\mathcal{O}_L$ is $\mathfrak{A}_{H_2}$-free and $\mathfrak{A}_{H_3}$-free.
\end{proof}

Now, we study the $\mathfrak{A}_{H_2}$-freeness of $\mathcal{O}_L$ for a couple of $n<0$ (recall that such an study for $H_3$ is completely analog). We know by Corollary \ref{sufffreenesswild} that $\mathcal{O}_L$ is $\mathfrak{A}_{H_2}$-free for $n=-2$.

\begin{example}\label{examplenonh2h3} Assume that $n=-6$, so $L=\mathbb{Q}(\sqrt{m},\sqrt{-6})$. Then, $\mathcal{O}_L$ is $\mathfrak{A}_{H_2}$-free. Additionally, assume that $m<0$. Then, $\mathcal{O}_L$ is $\mathfrak{A}_{H_3}$-free if and only if $3$ divides $m$.
\end{example}
\begin{proof}
By Proposition \ref{profreebiquadQ1}, in order to prove that $\mathcal{O}_L$ is $\mathfrak{A}_{H_2}$-free, it is enough to check that at least one of the equations $x^2-6y^2=\pm2d$ is solvable. Since $m$ is odd, it must be $d\in\{1,3\}$. For $d=1$, the equation $x^2-6y^2=-2$ has solution $(x,y)=(2,1)$, and for $d=3$, we are done by Corollary \ref{corofreenesswildsep} as $n=-2d$ (or alternatively, the equation $x^2-6y^2=-6$ has solution $(x,y)=(0,1)$).

On the other hand, if $m<0$, since $k=\frac{-6m}{d^2}>0$ we know by Proposition \ref{biquadQ1posit} that $\mathcal{O}_L$ is $\mathfrak{A}_{H_3}$-free if and only if $|n|=2d$. Since $n=-6$ and $d\in\{1,3\}$, this happens if and only if $d\neq1$.
\end{proof}

We have seen that the behaviour of $\mathcal{O}_L$ in the Hopf-Galois structures $H_2$ and $H_3$ is pretty similar, which is coherent with the fact that they are given by square root of numbers of the same class mod $4$. However, Example \ref{examplenonh2h3} gives a bunch of extensions $L/\mathbb{Q}$ such that $\mathcal{O}_L$ is free in one and it is not in the other; namely, $L=\mathbb{Q}(\sqrt{m},\sqrt{-6})$ with $m<0$, $m\equiv3\pmod 4$ and $m$ coprime with $6$.

Additionally, Example \ref{examplenonh2h3} also provides an example of extension $L/\mathbb{Q}$ for which $\mathcal{O}_L$ is $\mathfrak{A}_{H_2}$-free but it does not satisfy the conditions of Corollary \ref{corofreenesswildsep}. Indeed, if we choose $m=-1$, for $L=\mathbb{Q}(\sqrt{-1},\sqrt{-6})$, we have that $\mathcal{O}_L$ is $\mathfrak{A}_{H_2}$-free but $n\neq\pm2d$.

The situation is slightly different for $n=-10$.

\begin{example} Assume that $n=-10$, so $L=\mathbb{Q}(\sqrt{m},\sqrt{-10})$. Then, $\mathcal{O}_L$ is $\mathfrak{A}_{H_2}$-free if and only if $m$ and $n$ are not coprime. Additionally, assume that $m<0$. Then, $\mathcal{O}_L$ is also $\mathfrak{A}_{H_3}$-free if and only if $m$ and $n$ are not coprime.
\end{example}
\begin{proof}
By Proposition \ref{profreebiquadQ1}, it suffices to check  the equations $x^2-10y^2=\pm2d$. If $d=5$, then $n=-2d$, so $\mathcal{O}_L$ is $\mathfrak{A}_{H_2}$-free. If $d=1$, since $\pm2$ are not squares mod $10$, the equations are not solvable in $\mathbb{Z}$, so $\mathcal{O}_L$ is not $\mathfrak{A}_{H_2}$-free.

If $m<0$, then $k=\frac{-10m}{d^2}>0$, and $\mathcal{O}_L$ is $\mathfrak{A}_{H_3}$-free if and only if $|n|=2d$. Since $n=-10$ and $m$ is odd, $d\in\{1,5\}$, so $|n|=2d$ if and only if $d=5$.
\end{proof}

Finally, we study the solvability of the Pell equations determining the freeness in $H_i$ for $i\in\{1,2,3\}$. Since $n$ and $k$ are congruent mod $4$, we can exchange them in Proposition \ref{profreebiquadQ1}. This gives:

\begin{coro}
Let $m\equiv3\pmod 4$, $n\equiv2\pmod 4$, $d=\mathrm{gcd}(m,n)$ and $k=\frac{mn}{d^2}$.
\begin{itemize}
    \item[1.] At least one of the equations $x^2+my^2=\pm4d$ has solutions in $\mathbb{Z}$ if and only if so has at least one of the equations $x^2+my^2=\pm4\frac{m}{d}$.
    \item[2.] At least one of the equations $x^2+ny^2=\pm2d$ has solutions in $\mathbb{Z}$ if and only if so has at least one of the equations $x^2+ny^2=\pm2\frac{n}{d}$.
    \item[3.] At least one of the equations $x^2+ky^2=\pm2\frac{n}{d}$ has solutions in $\mathbb{Z}$ if and only if so has at least one of the equations $x^2+ky^2=\pm2\frac{m}{d}$.
\end{itemize}
\end{coro}
\begin{proof}
Let $L=\mathbb{Q}(\sqrt{m},\sqrt{n})$ and let $H_1$, $H_2$ and $H_3$ be its non-classical Hopf-Galois structures denoted as usual. By Proposition \ref{profreebiquadQ1}, at least one of the equations $x^2+my^2=\pm4d$ is solvable if and only if $\mathcal{O}_L$ is $\mathfrak{A}_{H_1}$-free. Now, since $n,k\equiv2\pmod 4$, we can exchange them, and we obtain that $\mathfrak{A}_{H_1}$-freeness is also equivalent to the solvability of at least one of the equations $x^2+my^2=\pm4\frac{m}{d}$, proving 1.

The argument for the second and the third statement is slightly different. If we exchange $n$ and $k$, then $\sqrt{n}$ gives the same Hopf-Galois structure $H_2$ but in third place, so the new criterion for the freeness in $H_2$ is given by exchanging $n$ and $k$ in the original criterion for the freeness in $H_3$, obtaining the equivalence stated. As for the proof of 3, after the replacement of $n$ by $k$, $\sqrt{k}$ gives the Hopf-Galois structure $H_3$ in second place, so to obtain the new criterion we exchange $n$ and $k$ in the original criterion for $H_2$.
\end{proof}

\subsection{Biquadratic extensions of second and third type}

Now, we consider wildly ramified biquadratic extensions of second type and tamely ramified biquadratic extensions. The main reason to join these two families is that the criteria for the freeness obtained in Propositions \ref{profreebiquadQ2} and \ref{profreebiquadQ3} can be reunited in the following:

\begin{pro}\label{profreebiquadQ23reform} Let $L/\mathbb{Q}$ be a biquadratic extension and let $a,b\in\mathbb{Z}$ be square-free integers with $a\equiv1\pmod 4$ such that $L=\mathbb{Q}(\sqrt{a},\sqrt{b})$. Let $H$ be the non-classical Hopf-Galois structure of $L/\mathbb{Q}$ given by $a$. 
Let $g=\mathrm{gcd}(a,b)$. Then, $\mathcal{O}_L$ is $\mathfrak{A}_H$-free if and only if at least one of the equations $x^2+ay^2=\pm2g$ is solvable.
\end{pro}

This indeed contains both results because if $b\equiv1\pmod 4$, then $L/\mathbb{Q}$ is tamely ramified, and otherwise, it is wildly ramified of second type.

As already mentioned, Truman obtained some results using the characterization for the freeness in tamely ramified biquadratic extensions. We summarize them in the following:

\begin{teo}[Truman]\label{tamebiquadtruman} Let $L/\mathbb{Q}$ be a tame biquadratic extension. Then:
\begin{itemize}
    \item[1.] If $L=\mathbb{Q}(\sqrt{p},\sqrt{q})$ with $p,q$ prime numbers such that $p\equiv q\equiv1\pmod 4$, then $\mathcal{O}_L$ is not $\mathfrak{A}_H$-free for all non-classical Hopf-Galois structures $H$.
    \item[2.] If $L=\mathbb{Q}(\sqrt{-p},\sqrt{-q})$ with $p,q$ prime numbers such that $p\equiv q\equiv3\pmod 4$, then $\mathcal{O}_L$ is free over the associated order in exactly two of the non-classical Hopf-Galois structures.
    \item[3.] If $L=\mathbb{Q}(\sqrt{-p},\sqrt{-qr})$ with $p,q,r$ prime numbers such that $p\equiv q\equiv3\pmod 4$ and $r\equiv1\pmod 8$, then $\mathcal{O}_L$ is free over the associated order in exactly one of the non-classical Hopf-Galois structures.
    \item[4.] $\mathcal{O}_L$ is not simultaneously free over the associated order in all non-classical Hopf-Galois structures.
\end{itemize}
\end{teo}

Proofs of these statements can be consulted in \cite[Example 7.1, Example 7.2, Example 7.3 and Theorem 7.4]{truman}. The first and the last ones are immediate consequences of the following result, which also holds for the wild case:

\begin{lema} Let $H$ be the non-classical Hopf-Galois structure of $L/\mathbb{Q}$ given by $\sqrt{a}$. If $a>1$, then $\mathcal{O}_L$ is not $\mathfrak{A}_H$-free.
\end{lema}
\begin{proof}
Assume that $\mathcal{O}_L$ is $\mathfrak{A}_H$-free, so there are $x,y\in\mathbb{Z}$ such that $x^2+ay^2=2g$, where $g=\gcd(a,b)$. 
Now, $x^2=2g-ay^2$ gives $ay^2\leq2g$, and since $y$ is non-zero, necessarily $a\leq2g$. Since $a$ is odd and divisible by $g$, necessarily $a=g$. But then $x^2+gy^2=2g$ gives   $y=\pm1$ and $x^2=g$, which is only possible if $g=1$, which contradicts $a=g$.
\end{proof}

Now, we are interested in the cases corresponding to $a<0$. We can simplify the criteria for the freeness in Proposition \ref{profreebiquadQ3} by discarding one of the equations.

\begin{lema}\label{lemasolveq}
\begin{itemize}
    \item[1.] If $D\equiv3\pmod 8$ and $N\equiv2\pmod 8$, then  $x^2-Dy^2=N$ is not solvable.
    \item[2.] If $D\equiv7\pmod 8$ and $N\equiv6\pmod 8$, then  $x^2-Dy^2=N$ is not solvable.
\end{itemize}
\end{lema}
\begin{proof}
In both cases, we check that the equation $x^2-Dy^2=N$ is not solvable in $\mathbb{Z}/8\mathbb{Z}$, so it is not solvable in $\mathbb{Z}$. Indeed,
since the squares mod $8$ are $0$, $1$ and $4$, if $D\equiv3\pmod 8$, the possible values of $x^2-Dy^2$ mod $8$ are $0$, $1$, $4$, $5$ and $6$. On the other hand, if $D\equiv7\pmod 8$, the possible values of $x^2-Dy^2$ mod $8$ are $0$, $1$, $2$, $4$ and $5$.
\end{proof}

\begin{pro}\label{proeasiertame} Let $L=\mathbb{Q}(\sqrt{a},\sqrt{b})$ be a biquadratic extension of $\mathbb{Q}$ with $a\equiv1\pmod 4$ and let $H$ be the non-classical Hopf-Galois structure of $L/\mathbb{Q}$ given by $\sqrt{a}$. Denote $g=\mathrm{gcd}(a,b)$.
\begin{itemize}
    \item[1.] If $a\equiv1\pmod 8$ and $g\equiv1\pmod 4$ or $a\equiv5\pmod 8$ and $g\equiv3\pmod 4$, then $x^2+ay^2=-2g$ is not solvable in $\mathbb Z$. Hence, $\mathcal{O}_L$ is $\mathfrak{A}_H$-free if and only if $x^2+ay^2=2g$ is solvable.
    \item[2.] If $a\equiv1\pmod 8$ and $g\equiv3\pmod 4$ or $a\equiv5\pmod 8$ and $g\equiv1\pmod 4$, then $x^2+ay^2=2g$ is not solvable. Hence, $\mathcal{O}_L$ is $\mathfrak{A}_H$-free if and only if $x^2+ay^2=-2g$ is solvable.
\end{itemize}
\end{pro}
\begin{proof}
\begin{itemize}
    \item[1.] If $a\equiv1\pmod 8$ and $g\equiv1\pmod 4$, then $-a\equiv7\pmod 8$ and $-2g\equiv6\pmod 8$, so the equation $x^2+ay^2=-2g$ is not solvable by the second statement of Lemma \ref{lemasolveq}. Identical conclusion is obtained if $a\equiv5\pmod 8$ and $g\equiv3\pmod 4$ by applying the first statement of Lemma \ref{lemasolveq}.
    \item[2.] It is completely analogous to the above.
\end{itemize}
\end{proof}

We can study easily the behaviour of $\mathcal{O}_L$ as $\mathfrak{A}_H$-module for low values of $|a|$.

\begin{example}\label{extamelowneg} Let $b\in\mathbb{Z}$ be a square-free integer and let $L=\mathbb{Q}(\sqrt{a},\sqrt{b})$ with $a\in\{-3,-7\}$. Let $H$ be the non-classical Hopf-Galois structure of $L/\mathbb{Q}$ given by $\sqrt{a}$. Then, $\mathcal{O}_L$ is $\mathfrak{A}_H$-free.
\end{example}
\begin{proof}
Let $g=\mathrm{gcd}(a,b)$. Since $a\equiv1\pmod 4$, by Proposition \ref{profreebiquadQ23reform}, $\mathcal{O}_L$ is $\mathfrak{A}_H$-free if and only if at least one of the equations $x^2+ay^2=\pm2g$ is solvable.

Assume that $a=-3$, so $g\in\{1,3\}$. For $g=1$, the equation $x^2-3y^2=-2$ is satisfied for $(x,y)=(1,1)$, and for $g=3$, the equation $x^2-3y^2=6$ is satisfied for $(x,y)=(3,1)$.

Now, assume that $a=-7$. In this case, we have $g\in\{1,7\}$. The equation $x^2-7y^2=2$ is satisfied for $(x,y)=(3,1)$, and the equation $x^2-7y^2=-14$ is satisfied for $(x,y)=(7,3)$.
\end{proof}

In general, the results in this part show that freeness is much less common in these cases than in the previous one,  biquadratic extensions of first type. The main reason is that, for those extensions, the presence of $4$, which is a square, in the independent term of the first equation in Proposition \ref{profreebiquadQ1} permitted to obtain solutions of the Pell equations under more general conditions.

Let us change the approach and focus on the Pell equations themselves. The consistency of Propositions \ref{profreebiquadQ2} and \ref{profreebiquadQ3} gives the following result on the solvability:

\begin{coro} Let $m,n$ be square-free integers with $m\equiv1\pmod 4$, $d=\mathrm{gcd}(m,n)$ and $k=\frac{mn}{d^2}$. Then:
\begin{itemize}
    \item[1.] At least one of the equations $x^2+my^2=\pm2d$ has solutions in $\mathbb{Z}$ if and only if so has at least one of the equations $x^2+my^2=\pm2\frac{m}{d}$.
\end{itemize}
If we additionally assume that $n\equiv1\pmod 4$, then:
\begin{itemize}
    \item[2.] At least one of the equations $x^2+ny^2=\pm2d$ has solutions in $\mathbb{Z}$ if and only if so has at least one of the equations $x^2+ny^2=\pm2\frac{n}{d}$.
    \item[3.] At least one of the equations $x^2+ky^2=\pm2\frac{n}{d}$ has solutions in $\mathbb{Z}$ if and only if so has at least one of the equations $x^2+ky^2=\pm2\frac{m}{d}$.
\end{itemize}
\end{coro}
\begin{proof}
Let $L=\mathbb{Q}(\sqrt{m},\sqrt{n})$ and let $H_1$, $H_2$, $H_3$ be its non-classical Hopf-Galois structures as before. Since $n\equiv k\pmod 4$, they can be exchanged in Propositions \ref{profreebiquadQ2} and \ref{profreebiquadQ3}. Moreover, when $n\equiv1\pmod 4$, $m$, $n$ and $k$ can be reordered indistinctly in Proposition \ref{profreebiquadQ3}.
\begin{itemize}
    \item[1.] We know that $\mathcal{O}_L$ is $\mathfrak{A}_{H_1}$-free if and only if at least one of the equations $x^2+my^2=\pm2d$ is solvable in $\mathbb{Z}$. Now, we exchange $n$ and $k$, so the $\mathfrak{A}_{H_1}$-freeness of $\mathcal{O}_L$ is also equivalent to the solvability in $\mathbb{Z}$ of at least one of the equations $x^2+my^2=\pm2\frac{m}{d}$. This implies immediately the statement.
    \item[2.] In this case, we use the $\mathfrak{A}_{H_2}$-freeness criterion as link for the two conditions, from which the result follows, as exchanging $m$ and $k$ in the equations $x^2+ny^2=\pm2d$ gives $x^2+ny^2=\pm2\frac{n}{d}$.
    \item[3.] Exchanging $m$ and $n$ in the equations $x^2+ky^2=\pm2\frac{n}{d}$ gives $x^2+ky^2=\pm2\frac{m}{d}$, so the result follows from the consistency of the criteria given for the $\mathfrak{A}_{H_3}$-freeness of $\mathcal{O}_L$.
\end{itemize}
\end{proof}

\section{$\mathfrak A_H$-locally free extensions that are not $\mathfrak A_H$-free}\label{sectexampleslocalglobal}

It is remarkable that, for the intended classes of extensions in this paper, our techniques allow to completely determine the freeness of the ring of integers in the different Hopf-Galois structures in the global context, i.e. for extensions of number fields. Normally, this is a tough problem which is studied through the local freeness.
Namely, if $L/K$ is an $H$-Galois extension of number fields, we say that $\mathcal{O}_L$ is $\mathfrak{A}_H$-locally free if $\mathcal{O}_{L,P}$ is $\mathfrak{A}_{H,P}$-locally free for every prime ideal $P$ of $\mathcal{O}_K$, where $$\mathcal{O}_{L,P}\coloneqq\mathcal{O}_L\otimes_{\mathcal{O}_K}\mathcal{O}_{K,P},$$ $$\mathfrak{A}_{H,P}\coloneqq\mathfrak{A}_H\otimes_{\mathcal{O}_K}\mathcal{O}_{K,P},$$ and $\mathcal{O}_{K,P}$ is the $P$-adic completion of $\mathcal{O}_K$.

In this section, we give account of the local context in the problem we have studied in order to illustrate that local freeness does not imply global freeness. Namely, we will give an example of quartic extension $L/\mathbb{Q}$ such that for a non-classical Hopf-Galois structure $H$ of $L/\mathbb{Q}$, $\mathcal{O}_L$ is $\mathfrak{A}_H$-locally free but $\mathcal{O}_L$ is not $\mathfrak{A}_H$-free. It will be easier (and possible) to restrict ourselves to the tame case. Indeed, among the quartic extensions of $\mathbb{Q}$, the tamely ramified ones are those for which all ramified rational primes are odd. Now, we have the following result due to Truman (see \cite[Theorem 6.2]{truman3}):

\begin{teo}\label{tamelocallyfree} Let $L/K$ be a tamely ramified abelian extension of number fields, and let $H=L[N]^G$ be a Hopf-Galois structure on $L/K$. Assume that $H$ is commutative. Then $\mathcal{O}_L$ is $\mathcal{O}_L[N]^G$-locally free.
\end{teo}

One can easily construct examples of biquadratic extensions with the already mentioned properties from the work of Truman. Indeed, if $L/\mathbb{Q}$ is a tamely ramified biquadratic extension, by Theorem \ref{tamebiquadtruman}, there is some Hopf-Galois structure $H$ on $L/\mathbb{Q}$ such that $\mathcal{O}_L$ is not $\mathfrak{A}_H$-free. However, we know by Theorem \ref{tamelocallyfree} that $\mathcal{O}_L$ is $\mathfrak{A}_H$-locally free.

Using our results, we may construct an example of quartic cyclic extension with this behaviour.

\begin{example}\normalfont We consider $L=\mathbb{Q}(\sqrt{65+4\sqrt{65}})$, which is cyclic of degree $4$ with $a=1$, $b=4$, $c=7$ and $d=65$, hence corresponds to Case 5 in Section \ref{sectcyclicquarticQ}. This number field corresponds to the identifier \cite[\href{https://www.lmfdb.org/NumberField/4.4.274625.2}{Number field 4.4.274625.2}]{lmfdb} in the LMFDB database. We can see that its ramified primes are $5$ and $13$, so it is tamely ramified and hence $\mathfrak{A}_H$-locally free for every Hopf-Galois structure $H$ on $L/\mathbb{Q}$.

Let us check that $\mathcal{O}_L$ is not (globally) $\mathfrak{A}_H$-free. By Theorem \ref{teofreenesscyclic}, $\mathcal{O}_L$ is $\mathfrak{A}_H$-free if and only if the quadratic form $[7,8,-7]$ represents $1$. But the cycle of indefinite reduced forms of this quadratic form is $$[7,8,-7],[-7,6,8],[8,10,-5],[-5,10,8],$$ which does not contain the principal form. Hence $[7,8,-7]$ does not represent $1$ and $\mathcal{O}_L$ is not $\mathfrak{A}_H$-free.
\end{example}

\section*{Acknowledgements} 

The authors want to thank Paul Truman and Nigel Byott for enriching discussions and their insightful remarks on the material covered; as well as the referee, for the same reasons and also for their suggestions in order to improve the quality and readability of this paper.

\end{document}